\documentclass[12pt,reqno]{amsart} 
\usepackage{amsmath,amsthm}
\usepackage{xypic}
\usepackage{enumerate}
\usepackage{lscape}

\usepackage{amscd}
\usepackage{amsfonts,amssymb,latexsym} 

\usepackage[all]{xy}
\xyoption{arc}
\CompileMatrices

\newcommand{\SC}{{\mathcal{C}}}
\newcommand{\SD}{{\mathcal{D}}}
\newcommand{\SE}{{\mathcal{E}}}
\newcommand{\SF}{{\mathcal{F}}}
\newcommand{\SM}{{\mathcal{M}}}
\newcommand{\SN}{{\mathcal{N}}}
\newcommand{\SO}{{\mathcal{O}}}
\newcommand{\SP}{{\mathcal{P}}}

\newcommand{\ST}{{\mathcal{T}}}
\newcommand{\SU}{{\mathcal{U}}}

\newcommand{\SZ}{{\mathcal{Z}}}
\newcommand{\TT}{\mathbb{T}}
\newcommand{\PP}{\mathbb{P}}
\newcommand{\CC}{\mathbb{C}}
\newcommand{\RR}{\mathbb{R}}
\newcommand{\ZZ}{\mathbb{Z}}
\newcommand{\Hom}{\operatorname{Hom}}
\newcommand{\End}{\operatorname{End}}
\newcommand{\Pic}{\operatorname{Pic}}
\newcommand{\id}{\operatorname{Id}}
\newcommand{\Aut}{\operatorname{Aut}}
\newcommand{\rk}{{\rm rank}}
\newcommand{\tr}{\operatorname{tr}}
\newcommand{\Mat}{\operatorname{Mat}}
\newcommand{\wt}{\widetilde}
\newcommand{\GL}{\operatorname{GL}}
\newcommand{\PGL}{\operatorname{PGL}}
\newcommand{\SSL}{\operatorname{SL}}
\newcommand{\op}{\operatorname}

\newcommand{\gitq}{/\!\!/}

\newtheorem{proposition}{Proposition}[section]
\newtheorem{theorem}[proposition]{Theorem}
\newtheorem{lemma}[proposition]{Lemma}
\newtheorem{corollary}[proposition]{Corollary}
\newtheorem{remark}[proposition]{Remark}

\numberwithin{equation}{section}

\title[Automorphisms of moduli spaces of framed bundles]{Automorphism group of
a moduli space of framed bundles over a curve}

\author[D. Alfaya]{David Alfaya}

\address{Instituto de Ciencias Matem\'aticas (CSIC-UAM-UC3M-UCM),
Nicol\'as Cabrera 15, Campus Cantoblanco UAM, 28049 Madrid, Spain}

\email{david.alfaya@icmat.es}

\author[I. Biswas]{Indranil Biswas}

\address{School of Mathematics, Tata Institute of Fundamental Research,
Homi Bhabha Road, Mumbai 400005, India}

\email{indranil@math.tifr.res.in}

\date{}



\begin{document}

\begin{abstract}
Let $X$ be a smooth complex projective curve, and let $x\,\in\, X$ be a 
point. We compute the automorphism group of the moduli space of framed vector bundles 
on $X$ of rank $r\, \geq\, 2$ with a framing over $x$. It is shown that this
automorphism group is generated by the following three: (1) pullbacks
using automorphisms of the curve $X$ that fix the marked point $x$, (2) tensorization with 
certain line bundles over $X$ and (3) the action of $\PGL_r(\CC)$ through composition with 
the framing.
\end{abstract}

\maketitle

\section{Introduction}

Framed bundles (also called vector bundles with a level structure) are pairs 
$(E,\,\alpha)$ consisting of a vector bundle $E$ of rank $r\, \geq\, 2$ and a nonzero linear 
map $\alpha:E_x\longrightarrow \CC^r$ from a fiber over a fixed point $x\in X$ to 
$\CC^r$; this $\alpha$ is called a framing. Framed bundles were first introduced by 
Donaldson as a tool to study the moduli space of instantons on $\RR^4$ \cite{Don84}. 
Latter on, Huybrechts and Lehn \cite{HL95modules, HL95pairs} defined framed modules 
as a common generalization of several notions of decorated sheaves, including framed 
bundles and Bradlow pairs. They described a general stability condition for framed 
modules and provided a geometric invariant theoretic construction for the moduli space
of framed modules.

A moduli space of framed bundles of rank $r$ is equipped with a canonical $\PGL_r(\CC)$-action; the 
action of any $[G]\,\in \,\PGL_r(\CC)$ sends a framed bundle $(E,\,\alpha)$ to $$[G]\cdot 
(E,\,\alpha)=(E,\,G\circ \alpha)\, .$$ In \cite{BGM10}, a Torelli type theorem was proved for 
the moduli space of framed bundles for small values of the stability parameter $\tau$ by 
studying this $\PGL_r(\CC)$-action. It was proved there that this action is essentially the 
only nontrivial $\PGL_r(\CC)$-action on the moduli space of framed bundles; the corresponding GIT-quotient was 
shown to be isomorphic to the moduli space of vector bundles.

We extend this result and prove that we can replace the restriction on $\tau$ in \cite{BGM10} (expressed in terms of
the rank) by a much weaker restriction which is independent of the rank.
We prove the following(see Theorem \ref{theorem:Torelli}):

\begin{theorem}\label{theorem:TorelliIntro}
Let $X$ and $X'$ be smooth complex projective curves of genus $g$ and $g'$ respectively, and let
$\tau$ and $\tau'$ be 
positive generic stability parameters such that $g\ge \max\{2+\tau,4\}$ and $g'\,\ge\, \max\{2+\tau',4\}$. Let $x\,\in\, X$ 
and $x'\,\in\, X'$ be marked points; let $\xi$ and $\xi'$ be line bundles over $X$ and $X'$ 
respectively. Let $\SF\,=\,\SF(X,x,r,\xi,\tau)$ be the moduli space of $\tau$-stable framed bundles 
over $(X,x)$ with rank $r\, \geq\, 2$ and determinant $\xi$. Similarly, set
$\SF'\,=\,\SF(X',x',r',\xi',\tau')$ to be the moduli space of $\tau'$-stable framed bundles
over $(X',\,x')$ with rank $r'\, \geq\, 2$ and determinant $\xi'$, and assume that there is an isomorphism 
$\Psi:\SF\stackrel{\sim}{\longrightarrow} \SF'$. Then $r=r'$ and there exists an isomorphism 
$\sigma:X\stackrel{\sim}{\longrightarrow} X'$ such that $\sigma(x)=x'$.
\end{theorem}

The version of this Torelli theorem for small parameters proven in \cite{BGM10} relied on the 
description of the automorphism group of the moduli space of vector bundles
obtained by Kouvidakis and Pantev
in \cite{KP95}. In order to deal with arbitrary stability parameters, we shall use a
generalization of this result that classifies the $k$-birational maps between the moduli spaces of 
vector bundles, i.e., isomorphisms between open subsets whose respective complements have 
codimension at least $k$. A similar $k$-birational classification was developed in the first 
author's Thesis for moduli spaces of parabolic vector 
bundles \cite{Alf18,AG19}. In Section \ref{section:k-bir} we incorporate some ideas from \cite{AG19} into the computation of the automorphism group of the moduli space of vector bundles done in \cite{BGM13} to obtain
the following result (see Theorem \ref{theorem:autoVB}):

\begin{theorem}\label{thorem:k-birIntro}
Let $X$ and $X'$ be smooth complex projective curves of genus $g\,\ge\, 4$ and $g'\,\ge\, 4$ respectively,
and let $\xi$ and $\xi'$ be line bundles over $X$ and $X'$ respectively. Let $\SM(X,r,\xi)$ denote
the moduli space of rank $r\, \geq\, 2$ semistable vector bundles over $X$ with determinant $\xi$.
Similarly, set $\SM(X',r',\xi')$ to be
the moduli space of rank $r'\, \geq\, 2$ semistable vector bundles over $X'$ with determinant $\xi'$. Let
$\SU\,\subset\, \SM(X,r,\xi)$ and $\SU'\subset \SM(X',r',\xi')$ be open subsets such that
\begin{eqnarray*}
\op{codim}(\SM(X,r,\xi)\backslash \SU,\SM(X,r,\xi))\ge 2\\
\op{codim}(\SM(X',r',\xi')\backslash \SU',\SM(X',r',\xi'))\ge 2
\end{eqnarray*}
If $\Phi\,:\, \SU\,\stackrel{\sim}{\longrightarrow}\, \SU'$ is an isomorphism then
\begin{enumerate}
\item $r\,=\,r'$,
\item there is an isomorphism $\sigma\,:\,X'\,\stackrel{\sim}{\longrightarrow}\, X$, and
\item there is a line bundle $L$ over $X$ such that either for every $E\,\in\, \SU$
$$\Phi(E) \,\cong \,\sigma^*(E\otimes L)$$
or for every $E\,\in\, \SU$
$$\Phi(E)\,\cong\, \sigma^*(E\otimes L)^\vee\, .$$
\end{enumerate}
\end{theorem}

Using Theorem \ref{thorem:k-birIntro}, we can approach the main aim of this paper. In Section \ref{section:auto} we compute the automorphism group of the moduli space of framed 
bundles with fixed determinant; more precisely, the following is proved (see Theorem 
\ref{thm:mainthm}):

\begin{theorem}\label{thm:thmIntro}
Let $X$ and $X'$ be smooth complex projective curves of genus $g\,\ge\, \max\{2+\tau,4\}$ and $g'\,\ge\, \max\{2+\tau',4\}$
with base points $x$ and $x'$. Assume that $\tau$ and $\tau'$ are generic stability
parameters and that there exists an isomorphism $\Psi$ between the moduli space of
$\tau$-semistable framed bundles on $X$ of rank $r\, \geq\, 2$ with fixed determinant $\xi$ and
framing over $x$ and the moduli space of $\tau'$-semistable framed bundles on $X'$
of rank $r'\, \geq\, 2$ with fixed determinant $\xi'$ and framing over $x'$. Then $r\,=\,r'$ and there
exists an isomorphism $\sigma\,:\,X'\,\longrightarrow\, X$ such that $\sigma(x')\,=\,x$, and the
isomorphism $\Psi$ is a combination of the following three types of transformations:
\begin{itemize}
\item pullback with respect to the isomorphism $\sigma:X'\longrightarrow X$
\item tensorization with a line bundle $L\in \Pic(X)$, and
\item action of $\PGL_r(\CC)$ defined by $[G]\cdot (E,\,\alpha)\,=\,(E,\,G\circ \alpha)$,
\end{itemize}
where $\sigma$ and $L$ satisfy the relation $\sigma^*(\xi \otimes L^{\otimes r}) \cong \xi'$. Moreover, $\tau$ and $\tau'$ belong to the same stability chamber, i.e., a framed bundle is $\tau$-stable if and only if it is $\tau'$-stable.
\end{theorem}

In particular, Theorem \ref{thm:thmIntro}
allows us to compute explicitly the structure of the automorphism 
group of a moduli space of framed bundles $\SF$ (see Corollary \ref{cor:maincor}):

\begin{corollary}
\label{cor:corIntro}
Let $\tau$ be a generic stability parameter and let $X$ be a curve of genus $g\ge \max\{2+\tau,4\}$ with a marked point $x\in X$. The automorphism group of $\SF=\SF(X,x,r,\xi,\tau)$ is
$$\Aut(\SF)\cong \PGL_r(\CC)\times \ST $$
for a group $\ST$ fitting in the short exact sequence
$$1\longrightarrow J(X)[r] \longrightarrow \ST \longrightarrow \Aut(X,x)
\longrightarrow 1\, ,$$
where $J(X)[r]$ is the $r$-torsion part of the Jacobian $J(X)$ of $X$ and
$$\Aut(X,x)\,=\,\{\sigma\,\in\, \Aut(X) \,\mid\, \sigma(x)\,=\,x\}\, .$$
\end{corollary}

\section{$k$-birational automorphisms of the moduli space of vector bundles}
\label{section:k-bir}

Let $M$ and $M'$ be two algebraic varieties, and let $k$ be a positive integer. A $k$-birational map is an
isomorphism $\varphi\,:\,U\,\longrightarrow\, U'$ between two open subsets $U\subset M$ and $U'\subset M'$ such that
$$\op{codim}(M\backslash U)\,\ge\, k$$
$$\op{codim}(M'\backslash U')\ge k\, .$$
We say that $M$ and $M'$ are $k$-birational if there exists a $k$-birational map between them. 
Observe that for $k=1$, a $1$-birational map is the same as a birational map, but there are 
some $k$-birational invariants which are not birational. For example, if $M$ and $M'$ are 
normal $2$-birational varieties, Hartogs' theorem proves that $H^0(M,\, {\mathcal O}_M)\,\cong\,
H^0(M',\, {\mathcal O}_{M'})$, but 
this would not be necessarily true if $M$ and $M'$ were just birational.

Let $X$ be a smooth complex projective curve. A vector bundle $E$ over $X$ is called stable 
(respectively, semistable) if for all proper subbundles $0\,\subsetneq\, E'\,\subsetneq\, E$ 
$$\frac{\text{degree}(E')}{\rk(E')} < \frac{\text{degree}(E)}{\rk(E)} \ \ 
(\text{respectively, }\frac{\text{degree}(E')}{\rk(E')} \leq \frac{\text{degree}(E)}{\rk(E)})$$
Let $\SM\,=\,\SM(X,r,\xi)$ denote the moduli space of semistable 
vector bundles over $X$ of rank $r\, \geq\, 2$ and determinant $\xi$, and let $\SM^s\, =\,\SM^s(X,r,\xi)$ be 
the Zariski open subset of it corresponding to stable vector bundles. See \cite[Theorem 2(B)]{NS65} or \cite[p.~635,
Theorem~2.8(B)]{Ma} for openness of the stability condition.

Let $L$ be a line bundle over $X$, and let $\sigma\,:\,X'\,\stackrel{\sim}{\longrightarrow} \,X$ be any
isomorphism between curves; take any $s\,\in\, \{1,\,-1\}$. We define the map
$$\ST_{\sigma,L,s}\,:\,\SM(X,r,\xi) \,\longrightarrow \,\SM(X',r,\sigma^*(\xi \otimes L^{\otimes r})^s)$$ as
\begin{eqnarray*}
\xymatrixrowsep{0.05pc}
\xymatrixcolsep{0.3pc}
\xymatrix{
\ST_{\sigma,L,+}&:&\SM(X,r,\xi) \ar[rrrr] &&&& \SM(X',r,\sigma^*(\xi\otimes L^{\otimes r}))\\
&& E \ar@{|->}[rrrr] &&&& \sigma^*(E\otimes L)
}
\end{eqnarray*}
for $s=1$, and
\begin{eqnarray*}
\xymatrixrowsep{0.05pc}
\xymatrixcolsep{0.3pc}
\xymatrix{
\ST_{\sigma,L,-}&:&\SM(X,r,\xi) \ar[rrrr] &&&& \SM(X',r,\sigma^*(\xi\otimes L^{\otimes r})^{-1})\\
&& E \ar@{|->}[rrrr] &&&& \sigma^*(E\otimes L)^\vee
} 
\end{eqnarray*}
for $s=-1$. In particular, if $\sigma:X\stackrel{\sim}{\longrightarrow} X$ is an
automorphism of a curve, and $\sigma^*(\xi \otimes L^{\otimes r})^s\cong \xi$, then the above defined
map 
$\ST_{\sigma,L,s}:\SM\longrightarrow \SM$ is an automorphism of the moduli space of
vector bundles such that $\ST_{\sigma,L,s}(\SM^s)=\SM^s$. Observe, however, that for $r=2$,
some of the previous transformations $\ST_{\sigma,L,s}$ are in fact redundant.

\begin{lemma}
\label{lemma:trivialrk2}
Let $r\,=\,2$. Then for every isomorphism
$$\ST_{\sigma,L,-}\,:\,\SM(X,r,\xi)\,\longrightarrow\, \SM(X',r,\sigma^*(\xi\otimes L)^{-1})$$
there exists a line bundle $L'$ on $X$ such that
$$\ST_{\sigma,L,-}\,=\,\ST_{\sigma,L',+}\, .$$
\end{lemma}

\begin{proof}
Since $\bigwedge^2E\,\cong\, \xi$ for every $E\,\in\, \SM(x,r,\xi)$, there is an isomorphism
$$E^\vee \,\cong\, E\otimes \xi^{-1}\, .$$
Consequently, for every $\sigma$ and $L$ we have
$$\sigma^*(E\otimes L)^\vee\,=\,\sigma^*(E^\vee \otimes L^{-1})\,\cong\, \sigma^*(E\otimes \xi^{-1} \otimes L^{-1})\, .$$
Then taking $L'=\xi^{-1}\otimes L^{-1}$ yields
$$\ST_{\sigma,L,-}\,=\,\ST_{\sigma,L',+}\, ,$$
thus proving the lemma.
\end{proof}

By \cite{KP95} and \cite{BGM13}, every automorphism of $\SM$ is given by a transformation of type 
$\ST_{\sigma,L,s}$ for suitable combinations of $\sigma$, $L$ and $s$. The objective of this 
section is to extend this result and prove that every $2$-birational map between moduli spaces of 
vector bundles must be of the form $\ST_{\sigma,L,s}$. This type of $2$-birational analogue has 
been obtained in \cite{AG19} for the moduli space of parabolic vector bundles and we can adapt the 
proof to the non-parabolic setup. The strategy will be similar to \cite{BGM13}, but we need to 
modify some technical steps following \cite{AG19} for the reason that we are now working with maps 
between different moduli spaces instead of automorphisms of a single moduli space.

First of all, observe that, contrary to the parabolic case, every $2$-birational map between moduli 
spaces of vector bundles actually extends to an isomorphism.

\begin{proposition}
\label{prop:extend2bir}
Let $\SU\,\subset\, \SM\,=\,\SM(X,r,\xi)$ and $\SU'\,\subset\, \SM'\,=\,\SM(X',r',\xi')$ be open subsets whose
respective complements have codimension at least $2$, and let $\Phi\,:\,\SU
\,\longrightarrow\,\SU'$ be an isomorphism. Then $\Phi$ extends uniquely to an isomorphism
$\overline{\Phi}\,:\,\SM(X,r,\xi)\,\longrightarrow\, \SM(X',r',\xi')$
\end{proposition}

\begin{proof}
Since $\SM$ is a normal variety, so is $\SU$. Therefore, as the codimension of $\SM\backslash \SU$
is at least $2$, every line bundle $L$ on $\SU$ extends uniquely to a line bundle on $\SM$, and
we have $\Pic(\SM)\,=\,\Pic(\SU)$. Similarly, $\Pic(\SM')\,=\,\Pic(\SU')$ and, therefore, $\Phi^*$
induces an isomorphism
$$\Pic(\SM)\,=\,\Pic(\SU)\,\cong\, \Pic(\SU')\,=\,\Pic(\SM')\, .$$
Take a sufficiently very ample line bundle $L'$ on $\SM'$, and consider the embedding
$\SM'\,\hookrightarrow\, \PP(H^0(\SM',\,L')^\vee)$. As $\Pic(\SU)\,=\,\Pic(\SM)\cong \ZZ$, and
$$\Phi^*\,:\,\Pic(\SU)\longrightarrow \Pic(\SU')$$ is an isomorphism, it follows that $\Phi^* L'|_{\SU'}$ extends
uniquely to a very ample line bundle $L$ on $\SM$, and we have an embedding $\SM\,\hookrightarrow\,
\PP(H^0(\SM,\,L)^\vee)$. Since the codimension of the complement of $\SU$ in $\SM$ is at least $2$, and
$\SM$ is normal, we have $H^0(\SM,\,L)\,=\,H^0(\SU,\,L)$. Similarly, we have
$$H^0(\SM',\,L')\,=\,H^0(\SU',\,L')\, .$$
Therefore, $\Phi$ induces an isomorphism
$$\PP(H^0(\SM,\,L)^\vee)\,=\,\PP(H^0(\SU,\,L)^\vee) \,\cong\, \PP(H^0(\SU',\,L')^\vee)\,=\,\PP(H^0(\SM',\,L')^\vee)$$
which sends $\SU$ to $\SU'$. As the closure of $\SU$ in $\PP(H^0(\SM,\,L)^\vee)$ is $\SM$ and
the closure of $\SU'$ in $\PP(H^0(\SM',\,L')^\vee)$ is $\SM'$, $\Phi$ extends uniquely to an
isomorphism between $\SM$ and $\SM'$.
\end{proof}

Therefore, we only need to classify isomorphisms between different moduli spaces of vector 
bundles. In order to do that, we can make use of the Higgs bundles.

A Higgs bundle is a pair $(E,\,\varphi)$ consisting of a vector bundle $E$ on $X$ and a
holomorphic section $\varphi\,\in\, H^0(X,\, \End_0(E)\otimes K_X)$ called
a Higgs field. A Higgs bundle $(E,\,\varphi)$ is stable (respectively, semistable) if for all proper subbundles $0\, \subsetneq \, E' \, \subsetneq \, E$ such that $\varphi(E')\,\subseteq E'\otimes K_X$
$$\frac{\text{degree}(E')}{\rk(E')} < \frac{\text{degree}(E)}{\rk(E)} \quad (\text{respectively, }\le).$$
Let $\SM_{\op{Higgs}}=\SM_{\op{Higgs}}(X,r,\xi)$ be the moduli space of semistable
Higgs bundles over $X$ of rank $r$ and determinant $\xi$ (see \cite{Hi1}, \cite{Hi2},
\cite{Si}).
By Serre duality, for every $E\in \SM^s$, the cotangent bundle $T_E^*\SM^s$ can be identified with
$H^1(X,\, \End_0(E))^\vee \,= \,H^0(X, \,\End_0(E)\otimes K_X)$. Therefore,
the total space of $T^*\SM^s$ is a Zariski open subset of $\SM_{\op{Higgs}}$.

The Higgs moduli space admits a map
$$H\,:\,\SM_{\op{Higgs}} \,\longrightarrow\, W\,:=\,\bigoplus_{i=2}^r H^0(X,\,K_X^i)$$
called the Hitchin map, defined as follows. Let $$\op{Tot}(K_X)\,=\,
\underline{\op{Spec}} \op{Sym}^\bullet (K_X^{-1})$$ be the total space of the canonical bundle. Let $\pi: \op{Tot}(K_X) \longrightarrow X$ be the projection and let $x\,\in\, H^0(\op{Tot}(K_X),\,\pi^*K_X)$ be the tautological section. For each Higgs bundle $(E,\varphi)$, consider the characteristic polynomial of $\varphi$
$$\det(x\cdot \id - \pi^*\varphi) = x^r+\widetilde{s_1} x^{r-1} + \ldots + \widetilde{s_r} \, .$$
Then there exist unique sections $$s_i\,=\,\tr(\wedge^i\varphi) \,\in\, H^0(X,\,K_X^i)$$ such
that $\widetilde{s_i}=\pi^*s_i$. Observe that $\tr(\varphi)=0$
as we are working with fixed determinant, so $s_1=0$. The Hitchin map is then
\begin{equation}\label{hm}
H(E,\varphi)\,=\,(s_i)_{i=2}^r\,\in\, W\,:=\,\bigoplus_{i=2}^r H^0(X,\,K_X^i)
\end{equation}
and we call $W$ the Hitchin space. Given a point $s=(s_2,\ldots,s_r)\in W$, the zeroes of the equation
$$x^r+s_2(z)x^{r-2}+\ldots+s_r(z)\,=\,0 \, ,$$
where $z$ is a coordinate for $X$ and $x$ is the tautological vertical coordinate in
$\op{Tot}(K_X)$, define a curve $X_s\,\subset\, \op{Tot}(K_X)$, which is always an $r$-to-$1$ cover of
$X$. We call $X_s$ the spectral curve associated to $s\,\in\, W$. Let $\SD\,\subsetneq\, W$ be the
divisor of the Hitchin space consisting of points $s\in W$ whose spectral curve $X_s$ is
singular. We call $H^{-1}(\SD)$ the Hitchin discriminant.

\begin{proposition}[{\cite[Theorem 3.3]{BGM13}}]
\label{prop:recoverDiscriminant}
The Hitchin discriminant $H^{-1}(\SD)$ is the closure of the union of the complete rational curves in $\SM_{\op{Higgs}}$.
\end{proposition}

Moreover, the following propositions allow us to recover the Hitchin map of the moduli space up 
to an automorphism of the Hitchin space and identify a certain subvariety inside a piece of the 
Hitchin space from which we can obtain the isomorphism class of the curve.

\begin{proposition}[{\cite[Lemma 4.1]{BGM13}}]
\label{prop:HitchinGlobalFunctions}
When $g\ge 2$ the global algebraic functions $H^0(T^*\SM,\, {\mathcal O})$ produce a map
$$\widetilde{h}:T^*\SM \longrightarrow \op{Spec}(H^0(T^*\SM,\, {\mathcal O}))\cong W\cong \CC^m$$
which is the restriction of the Hitchin map to $T^*\SM$ up to an isomorphism of $\CC^m$, where
$m=\dim W$. Moreover, for the action $\lambda\cdot (E,\theta)\,=\,
(E,\lambda\theta)$ of $\CC^*$ on $T^*\SM$, there is a unique $\CC^*$ action on $W$
such that $\widetilde{h}$ is $\CC^*$-equivariant.
\end{proposition}

Let $W_i\,=\,H^0(X,\, K_X^i)$, and let $\pi_i\,:\,W\,=\,\bigoplus_{i=2}^r W_i\,\twoheadrightarrow\, W_i$ be the 
projection.

\begin{proposition}[{\cite[Proposition 4.2]{BGM13}}]
\label{prop:recoverDualVariety}
The intersection $\SC=\SD\cap W_r$ is irreducible, and $$\PP(\SC)\,\subset\, \PP(W_r)$$
is the dual variety of $X\,\subset\, \PP(W_r^\vee)$ for the embedding given by the linear series $|K_X^r|$.
\end{proposition}

Now, we can proceed as in \cite[Theorem 4.3]{BGM13} in order to recover the Hitchin map and obtain a Torelli type theorem for the moduli space. If $\SM=\SM(X,r,\xi)$ and $\SM'=\SM(X',r',\xi')$ are two moduli spaces of vector bundles and $\Phi\,:\,\SM\,\longrightarrow\, \SM'$ is an isomorphism, then $\Phi$ induces an isomorphism between the cotangent bundles $d(\Phi^{-1}):T^*\SM \longrightarrow T^*\SM'$ which, by Proposition \ref{prop:HitchinGlobalFunctions}, must induce
a $\CC^*$-equivariant isomorphism
$$f\,:\,W\,\cong\, \op{Spec}(H^0(T^*\SM,\, {\mathcal O})) \,\stackrel{\sim}{\longrightarrow}\, \op{Spec}
(H^0(T^*\SM',\, {\mathcal O}))\,\cong\, W'$$ such that the following diagram commutes
\begin{eqnarray*}
\xymatrixcolsep{3pc}
\xymatrix{
T^*\SM \ar[r]^{d(\Phi^{-1})} \ar[d]_{H} & T^*\SM' \ar[d]^{H'}\\
W \ar[r]^{f} & W'
}
\end{eqnarray*}
On the other hand, as $d(\Phi^{-1})$ is an isomorphism, it sends complete rational curves to complete rational curves. Therefore, by Proposition \ref{prop:recoverDiscriminant}, if we call $\SD$ and $\SD'$ the images of the discriminant loci in $W$ and $W'$ respectively, then $f(\SD)=\SD'$. 

The $\CC^*$-action on $W$ (respectively, $W'$) induces a stratification of the Hitchin space into 
subspaces corresponding to the points whose rate of decay is at least $|\lambda|^i$ for each 
$i=2,\ldots,r$. The map $f$ is $\CC^*$-equivariant, so it must preserve the stratification. In 
particular, the number of steps in the filtration must be the same, so $r=r'$. Moreover, $f$ must 
preserve the space of maximum decay, $W_r$, so $f(W_r)\,=\,W_r'$ and, as the $\CC^*$-action is 
homogeneous on $W_r$ of maximum decay, $f|_{W_r}$ must be linear (c.f. \cite[Lemma 7.1]{AG19}). 
Set $\SC\,=\,\SD\cap W_r$ and $\SC'\,=\,\SD'\cap W_r'$. We know that $f(\SD)\,=\,\SD'$, so $f(\SC)\,=\,\SC'$. By 
Proposition \ref{prop:recoverDualVariety}, the dual variety of $\PP(\SC)$ in $\PP(W_r)$ is $X$ 
and the dual variety of $\PP(\SC')$ in $\PP(W_r')$ is $X'$. As $f|_{W_r}$ is linear and 
$f(\SC)\,=\,\SC'$, we conclude that $f$ induces an isomorphism $f^\vee\,:\,\PP(W_r^\vee)\,\longrightarrow\,
\PP((W_r')^\vee)$ sending $X$ to $X'$.

Therefore, composing $\Phi$ with the pullback $\ST_{\sigma,\SO_X,+}$ if necessary,
henceforth we can assume without loss of generality that
\begin{itemize}
\item $r\,=\,r'$,

\item $X\,=\,X'$, and

\item the automorphism $\sigma\,:\,X\,\longrightarrow\, X$ induced by $f^\vee$ is the identity map.
\end{itemize}
These in particular imply that $W=W'$ and $\SD=\SD'$.

Observe that Proposition \ref{prop:HitchinGlobalFunctions} does not allow us to recover the 
Hitchin map completely, as it only recovers the Hitchin space as an affine variety, without the 
linear structure. Moreover, the $\CC^*$-action induces a stratification of $W$, but, in 
principle, the decomposition $W\,=\,\bigoplus_{i=2}^r W_i$ might not be preserved. However, 
Proposition \ref{prop:recoverDiscriminant} allows us to recover the image of the discriminant 
$\SD\subset W$ in addition to the $\CC^*$-action and the following Lemma proves that the rich 
interaction of the geometry of the divisor $\SD$ in conjunction with the $\CC^*$-action is 
sufficient to recover the linear structure and the decomposition of $W$.

\begin{lemma}
\label{lemma:recoverHitchin}
Let $f\,:\,W\,\longrightarrow\,W$ be a $\CC^*$-equivariant algebraic isomorphism such that $f(\SD)=\SD$. Then for
every $i>1$, there exists a linear automorphism $f_i\,:\,W_i\,\longrightarrow\, W_i$ such that the diagram
\begin{eqnarray*}
\xymatrixcolsep{4pc}
\xymatrix{
W \ar[r]^{f} \ar[d]_{\pi_i} & W \ar[d]^{\pi_i} \\
W_i \ar[r]^{f_i} & W_i
}
\end{eqnarray*}
is commutative.
\end{lemma}

\begin{proof}
An analogous lemma was proved for the parabolic Hitchin space in \cite[Lemma 7.11]{AG19} and the 
same proof works here by simply setting the parabolic divisor to be the zero divisor. 
\end{proof}

\begin{lemma}
\label{lemma:preserveNilpotent}
Let $f_r\,:\,W_r\,\longrightarrow\, W_r$ be the linear automorphism constructed in
Lemma \ref{lemma:recoverHitchin}. Then 
$$f_r(H^0(X,\, K_X^r(-kx_0)))\,=\,H^0(X,\, K_X^r(-kx_0))$$
for every $k\,>\,0$ and every $x_0\,\in\, X$.
\end{lemma}

\begin{proof}
We have already showed that the map $f_r\,:\,W_r\,\longrightarrow\, W_r$ is linear and preserves $\SC
\,=\,W_r\cap \SD$. Moreover, we know that $\PP(\SC)\,\subset\, \PP(W_r)$ is the dual variety to $X\,\subset
\,\PP(W_r^\vee)$ and, without loss of generality, we assumed that the dual map $f_r^\vee\,:\, \PP(W_r^\vee)
\,\longrightarrow\, \PP(W_r^\vee)$ induces the identity map on $X$. In particular, for every $x_0\,\in\, X$, the
map $f_r^\vee$ must preserve all the osculating $k$ spaces at $x_0$, so the linear map $f$ must
preserve the subspaces $H^0(X,\, K_X^r(-kx_0))$.
\end{proof}

Now that we have recovered the linear structure of the Hitchin space and the subspaces 
$H^0(X,\, K_X^r(-kx_0))\,\subset\, W_r$, we can use the following Proposition from \cite{BGM13}.

\begin{proposition}[{\cite[Proposition 5.1]{BGM13}}]
\label{prop:recoverEndSections}
Fix a generic stable bundle $E\in \SM(X,r,\xi)$, and consider the map
$$h_r\,:\,T^*_E\SM\,\cong\, H^0(X,\, \End_0(E)\otimes K_X) \,\longrightarrow\, W_r$$
constructed as the composition of the Hitchin map with the projection $W\twoheadrightarrow W_r$. Then for every $x_0\in X$
$$H^0(X,\, \End_0(E)\otimes K_X(-x_0)) \,=\, \{\psi\,\in\, H_{x_0}\,\mid\,
h_r(\psi+\varphi)\,\in\, H_{x_0}\ \ \forall\ \varphi\,\in\, h_r^{-1}(H_{x_0}) \}\, ,$$
where $H_{x_0}\,=\,H^0(X,\, K_X^r(-x_0))\,\subset\, W_r$.
\end{proposition}

If $E$ and $E'$ are generic stable vector bundles such that $\Phi(E)\,=\,E'$, the map
$$d(\Phi^{-1})\,:\,H^0(X,\,\End_0(E)\otimes K_X) \,\longrightarrow \,H^0(X,\, \End_0(E') \otimes K_X)$$
satisfies the condition
$$d(\Phi^{-1})(H^0(X,\, \End_0(E) \otimes K_X(-x_0))) \,=\, H^0(X,\, \End_0(E') \otimes K_X(-x_0))$$
for every $x_0\,\in\, X$.

\begin{corollary}
\label{cor:recoverNilpotent}
Suppose that $g\ge 4$. For every $x\,\in\, X$, and every $i>1$, the linear subspace
$$H^0(X,\, K_X^i(-x))\,\subseteq\, W_i$$
is preserved by the linear map $f_i\,:\,W_i\,\longrightarrow\, W_i$.
\end{corollary}

\begin{proof}
We can adapt \cite[Lemma 7.18]{AG19} to the compact case. Let $\SU\,\subset\,\SM$ be the open
nonempty subset of vector bundles $E\,\in \,\SM$ such that both $E$ and $\Phi(E)$ are stable and
also generic in the sense of Proposition \ref{prop:recoverEndSections}. Let $\SU'=\Phi(\SU)$.
Since $g\ge 4$ (by assumption), we have
$$r\deg(K_X(-x))=r(2g-3) \ge 2(2g-3)>2g\, .$$
Therefore, applying \cite[Lemma 3.2]{AG19} to $L\,=\,
K_X(-x)$ and the open subsets $\SU$ and $\SU'$ we obtain that
$$\bigoplus_{i=2}^r H^0(X,\, K_X^i(-ix))\,\subseteq\, W$$
is the linear subspace of $W$ generated by the images
$$H(H^0(X,\, \End_0(E)\otimes K_X(-x)))\, \subset\, W\, ,$$
where $H$ is the map in \eqref{hm},
when $E$ runs over $\SU$ or
by the images $$H'(H^0(X,\, \End_0(E')\otimes K_X(-x)))$$ for $E'\,\in \,\SU'$.

By Proposition \ref{prop:recoverEndSections}, for every $E\,\in\, \SU$ with $$E'\,=\,\Phi(E)\in \SU'\, ,$$ 
the image of $H^0(X,\, \End_0(E)\otimes K_X(-x))$ by $d(\Phi^{-1})$ is $H^0(X,\, \End_0(E')\otimes 
K_X(-x))$. As $\Phi(\SU)=\SU'$, the union of the images $H(H^0(X,\, \End_0(E)\otimes K_X(-x)))$ for 
$E\in \SU$ is the same as the union of the images $H'(H^0(X,\,\End_0(E')\otimes K_X(-x)))$ for $E'\in 
\SU'$, so $f\,:\,W\,\longrightarrow\,W$ preserves the subspace $\bigoplus_{i=2}^r H^0(X,\,K_X^i(-ix))\,\subseteq\, W$.

By Lemma \ref{lemma:recoverHitchin}, the map $f\,:\,W\,\longrightarrow\, W$ is diagonal with respect to the 
decomposition $W=\bigoplus_{i=2}^r W_i$, so each individual piece $H^0(X,\,K_X^i(-ix))$ must be 
preserved by $f_i:W_i\longrightarrow W_i$.

For $i\,>\,1$, the curve $X$ is embedded in $\PP(W_i^\vee)$ via the linear system $|K_X^i|$, and for every $x\in X$, its osculating $i$-space is precisely
$$\op{Osc}_i(x)\,=\,\PP \left (\ker(H^0(K_X^i)^\vee \longrightarrow H^0(K_X^i(-ix))^\vee)\right)\, .$$
Therefore, the map $f_i^\vee \,:\, \PP(W_i^\vee)\,\longrightarrow\, \PP(W_i^\vee)$ preserves $\op{Osc}_i(x)$ for
all $x\in X$. By \cite[Lemma 7.19]{AG19} (which is based on an argument in
\cite[p.~1250052-23]{BGM12}), if a map between projective spaces preserves the osculating
function of an embedded smooth projective curve $\op{Osc}_i\,:\,X\,\longrightarrow\,\op{Gr}(i+1,m+1)$, it must
preserve all the osculating functions for each $i\,\ge\, 1$. In particular, $\op{Osc}_1$
is preserved, so $f_i$ must preserve the hyperplanes $H^0(X,\, K_X^i(-x))\,\subset\, H^0(X,\, K_X^i)$.
\end{proof}

Given a vector bundle $E$, for each $x\in X$ we can consider the nilpotent cone spaces
$$\SN_{E,x} = \left \{A\in \End_0(E)\otimes K_X|_x \, \middle | \, A^r=0 \right \} \subset \End_0(E)\otimes K_X|_x$$
Letting $x\in X$ vary, we construct a subscheme $\SN_E\,\subset\, \End_0(E)\otimes K_X$ over
$X$ that we shall call the nilpotent cone bundle. 

\begin{proposition}
\label{prop:nilpotentCone}
If $\Phi:\SM\longrightarrow \SM'$ is an isomorphism, and $E\in \SM$ is a generic point whose image
is $E'\in \SM'$, then there is an isomorphism $$\Phi_{\End}:\End_0(E)\otimes K_X
\longrightarrow \End_0(E') \otimes K_X$$ of vector bundles over $X$ which preserves
the nilpotent cone bundles, i.e., $\Phi_{\End}(\SN_E)=\SN_{E'}$.
\end{proposition}

\begin{proof}
Let $E\,\in\, \SM$ be a generic vector bundle. Consider the sub-bundle $\SE$ of the trivial vector bundle
$H^0(X,\, \End_0(E)\otimes K_X)\times X\,\longrightarrow\, X$ whose fiber over each point $x\,\in\, X$
is $H^0(X,\, \End_0(E)\otimes K_X(-x))$. By \cite[Remark 2.3]{BGM13}, for a generic vector bundle $E\,\in\, \SM$ we have
$H^0(X,\, \End_0(E)(x))\,=\,0$ for every $x\in X$. Therefore, $$H^1(X,\, \End_0(E)\otimes K_X(-x))\,=\,
H^0(X,\, \End_0(E)(x))^\vee\,=\,0$$ and we have a short exact sequence
$$0\longrightarrow H^0(X,\, \End_0(E)\otimes K_X(-x)) \longrightarrow H^0(X,\, \End_0(E)\otimes K_X) \longrightarrow \End_0(E)\otimes K_X|_x \longrightarrow 0\, .$$
Thus, we have a short exact sequence
$$0\longrightarrow \SE \longrightarrow H^0(X,\, \End_0(E)\otimes K_X)\otimes \SO_X \longrightarrow \End_0(E)\otimes K_X \longrightarrow 0$$
Similarly, assuming that $E'=\Phi(E)$ is generic, let $\SE'$ be the vector bundle fitting in the short exact sequence
$$0\longrightarrow \SE' \longrightarrow H^0(X,\, \End_0(E')\otimes K_X)\otimes \SO_X \longrightarrow \End_0(E')\otimes K_X \longrightarrow 0\, .$$
By Proposition \ref{prop:recoverEndSections}, the isomorphism
$$d(\Phi^{-1})\,:\,H^0(X,\, \End_0(E)\otimes K_X)\otimes \SO_X \,\longrightarrow \,H^0(X,\, \End_0(E')
\otimes K_X)\otimes \SO_X$$ sends the image of $\SE$ to $\SE'$, so it must induce an isomorphism $\Phi_{\End}$ on the quotients making the following diagram commute
\begin{eqnarray*}
\xymatrix{
0 \ar[r] & \SE \ar[r] & H^0(X,\,\End_0(E)\otimes K_X)\otimes \SO_X \ar[r] \ar[d]^{d(\Phi^{-1})} & \End_0(E)\otimes K_X \ar[d]^{\Phi_{\End}} \ar[r] & 0\\
0 \ar[r] & \SE' \ar[r] & H^0(X,\, \End_0(E')\otimes K_X)\otimes \SO_X \ar[r] & \End_0(E')\otimes K_X \ar[r] & 0
}
\end{eqnarray*}

Moreover, for generic $E\in \SM$, the preimage of the nilpotent cone $\SN_{E,x}$ under the surjective map
$H^0(X,\,\End_0(E) \otimes K_X) \longrightarrow \End_0(E)\otimes K_X|_x$
is
$$N_{E,x} = \left \{ \psi \in H^0(X,\, \End_0(E)\otimes K_X) \, \middle | \, \forall i>1 \, h_i(\psi) \in
H^0(X,\,K_X^i(-x)) \right \}\, .$$

Observe that, by Corollary \ref{cor:recoverNilpotent}, if $E$ and $E'=\Phi(E)$ are generic, then $d(\Phi^{-1})(N_{E,x})=N_{E',x}$. Therefore, we have $\Phi_{\End}(\SN_E)=\SN_{E'}$.
\end{proof}

\begin{theorem}
\label{theorem:autoVB}
Let $X$ and $X'$ be two smooth projective curves of genus $g\,\ge\, 4$ and $g'\,\ge\, 4$ respectively, and
let $\xi$ and $\xi'$ be line bundles over $X$ and $X'$ respectively. Let $\SU\,\subset\, \SM(X,r,\xi)$
and $\SU'\,\subset \,\SM(X',r',\xi')$ be Zariski open subsets such that
\begin{eqnarray*}
\op{codim}(\SM(X,r,\xi)\backslash \SU,\SM(X,r,\xi))\,\ge\, 2\\
\op{codim}(\SM(X',r',\xi')\backslash \SU',\SM(X',r',\xi'))\,\ge\, 2.
\end{eqnarray*}
If $\Phi\,:\, \SU\,\stackrel{\sim}{\longrightarrow}\, \SU'$ is an isomorphism, then
\begin{enumerate}
\item $r\,=\,r'$,
\item there exists an isomorphism $\sigma\,:\,X'\,\stackrel{\sim}{\longrightarrow}\, X$, and
\item there exist a line bundle $L$ over $X$ and $s\in \{\pm 1\}$ such that
$$\Phi(E) \,\cong\, \ST_{\sigma,L,s}(E)$$
for every $E\,\in\, \SU$.
\end{enumerate}
\end{theorem}

\begin{proof}
By Proposition \ref{prop:extend2bir}, we can assume that $$\SU\,=\,\SM(X,r,\xi)\,=\,\SM\ \ \text{ and }\ \
\SU'\,=\,\SM(X',r',\xi')\,=\,\SM'\, .$$ By the Torelli theorem \cite{KP95,BGM13}, $r=r'$ and $X\cong X'$. Let $W$ and $W'$ be the respective Hitchin spaces for $\SM$ and $\SM'$, and let us denote by $\SD\subset W$
and $\SD'\subset W'$ the divisors corresponding to singular spectral curves. Earlier in this
section, it was shown that the isomorphism $\Phi$ induces a $\CC^*$-equivariant isomorphism
$$f\,:\,W\,\cong \,\op{Spec}(H^0(T^*\SM,\, {\mathcal O})) \,\stackrel{\sim}{\longrightarrow}
\,\op{Spec}(H^0(T^*\SM',\, {\mathcal O}))\cong W'$$ such that the following diagram commutes
\begin{eqnarray*}
\xymatrixcolsep{3pc}
\xymatrix{
T^*\SM \ar[r]^{d(\Phi^{-1})} \ar[d]_{H} & T^*\SM' \ar[d]^{H'}\\
W \ar[r]^{f} & W'
}
\end{eqnarray*}
and $f(\SD)=\SD'$. Moreover, as $f$ is $\CC^*$-equivariant, it must preserve the subspace of maximum decay, so $f(W_r)=W_r'$. As a consequence of Proposition \ref{prop:recoverDualVariety}, $f|_{W_r}$ induces an isomorphism $\sigma:X\longrightarrow X'$. Substituting $\Phi$ by
$\ST_{\sigma,\SO_X,1}\circ \Phi$, we reduce the proof of the theorem to the case where $X=X'$
and $\sigma=\id$.

Now, we can apply Lemma \ref{lemma:recoverHitchin} to conclude that there exist linear
automorphisms $f_i:W_i\longrightarrow W_i$ such that the following diagram commutes
\begin{eqnarray*}
\xymatrixcolsep{4pc}
\xymatrix{
W \ar[r]^{f} \ar[d]_{\pi_i} & W \ar[d]^{\pi_i} \\
W_i \ar[r]^{f_i} & W_i
}
\end{eqnarray*}
for every $i>1$.
Moreover, by Lemma \ref{lemma:preserveNilpotent}, for each $x\in X$ and every $i>1$, the map $f_i$ preserves the hyperplane $H^0(X,\, K_X^i(-x))\,\subset\, W_i$.
Let us consider the open nonempty subset $\SU\subset \SM$ consisting of generic stable vector bundles $E\in \SM$ such that $\Phi(E)\in \SM'$ is also generic and stable. Let $\SU'=\Phi(\SU)$. By Proposition \ref{prop:nilpotentCone}, the
map $$d(\Phi^{-1})\,:\,T_E^*\SM\,\cong\, H^0(X,\,\End_0(E)\otimes K_X) \, \longrightarrow\, T_{E'}^*\SM'
\,\cong\, H^0(X,\,\End_0(E')\otimes K_X)$$ induces an isomorphism of nilpotent cone bundles over $X$
\begin{eqnarray*}
\xymatrix{
\SN_E \ar[r] \ar[d] & \SN_{E'} \ar[d] \\
X \ar@{=}[r] & X
}
\end{eqnarray*}
By \cite[Lemma 5.2]{BGM13}, this isomorphism induces an isomorphism of the flag bundles associated to $E$ and $E'$
\begin{eqnarray*}
\xymatrix{
\op{Fl}(E) \ar[r] \ar[d] & \op{Fl}(E') \ar[d] \\
X \ar@{=}[r] & X
}
\end{eqnarray*}
The isomorphism induced between the global vertical fields of both flag bundles corresponds to an isomorphism of the Lie algebra bundles
\begin{eqnarray*}
\xymatrix{
\End_0(E) \ar[r] \ar[d] & \End_0(E') \ar[d] \\
X \ar@{=}[r] & X
}
\end{eqnarray*}
By \cite[Lemma 5.4]{BGM13}, if $\End_0(E)$ and $\End_0(E')$ are isomorphic as Lie algebra bundles, then there exist a line bundle $L$ over $X$ and $s\in \{\pm 1\}$ such that 
\begin{equation}
\label{eq:theorem1}
E'=\ST_{\id,L,s}(E)\, .
\end{equation}
Therefore, we have shown that for each $E \in \SU$ there exist $L$ and $s$ such that
$\ST_{\id,L,s}^{-1} \circ \Phi (E)\, =\, E$. We need to prove that there exists a pair $(L,s)$ such that 
$$\Phi(E)\,=\,\ST_{\id,L,s}(E)$$ for every $E\,\in\, \SM$. Observe that taking determinant of both sides of
\eqref{eq:theorem1} we obtain that
$$\xi'\cong (\xi\otimes L^r)^s\, .$$
Let
$$\TT_{\xi,\xi'}=\left \{(L,s) \, \middle | \, L\in \Pic(X), \, s\in \{\pm 1\}, \, \xi'\cong (\xi\otimes L^r)^s \right \}\, .$$
Then, we have shown that for every $E\in \SU$, there exists $(L,s)\in \TT_{\xi,\xi'}$ such
that $E\in \op{Fix}(\ST_{\id,L,s}^{-1}\circ \Phi)$. In particular, we have proved that
$$\SU = \bigcup_{(L,s)\in \TT_{\xi,\xi'}} \op{Fix}(\ST_{\id,L,s}^{-1}\circ \Phi) \cap \SU\, .$$
The $r$-torsion part of the Jacobian of $X$ is finite, so the set $\TT_{\xi,\xi'}$ is finite. The set of fixed points of an automorphism is closed and $\ST_{\xi,\xi'}$ is finite, so $\SU$ is a finite union of closed subsets. As $\SM$ is irreducible,
the open subset $\SU$ is also irreducible, so there exists $(L,s)\in \ST_{\xi,\xi'}$ such that
$\SU= \op{Fix}(\ST_{\id,L,s}^{-1}\circ \Phi)$.

Therefore, $\Phi(E)\,=\,\ST_{\id,L,s}(E)$ for every $E\,\in \,\SU$. As $\ST_{\id,L,s}$ clearly extends 
to entire $\SM$, and $\SU$ is dense in $\SM$, the proof of the theorem is complete.
\end{proof}

\begin{corollary}
Let $X$ and $X'$ be smooth projective curves of genus at least $4$. Then the following are equivalent
\begin{enumerate}
\item The moduli spaces $\SM(X,r,\xi)$ and $\SM(X',r',\xi')$ are $2$-birational.
\item The moduli spaces $\SM(X,r,\xi)$ and $\SM(X',r',\xi')$ are isomorphic.
\item The curves $X$ and $X'$ are isomorphic, $r\,=\,r'$ and $\deg(\xi') \,\equiv\, \pm \deg(\xi)\ \pmod{r}$.
\end{enumerate}
\end{corollary}

\begin{proof}
(1) and (2) are equivalent by Proposition \ref{prop:extend2bir}. Let
$$\Phi\,:\,\SM(X,r,\xi) \,\longrightarrow \,\SM(X',r',\xi')$$ be an isomorphism.
By Theorem \ref{theorem:autoVB}, if the moduli spaces are isomorphic, $X\,\cong\, X'$, $r\,=\,r'$ and
there exist an isomorphism
$\sigma\,:\,X\,\longrightarrow\, X'$, a line bundle $L$ over $X$ and $s\,\in\, \{\pm 1\}$ such that
$$\Phi(E)\,=\,\ST_{\sigma,L,s}(E)$$ for every $E\,\in\, \SM(X,r,\xi)$. Then, in particular, we must
have $\xi'\,=\,\sigma^*(\xi\otimes L^r)^s$. Computing degrees we have
$\deg(\xi')\,=\,s(\deg(\xi)+r\deg(L))$ so the desired relation is obtained.

On the other hand, suppose that $X=X'$ and $r=r'$. If $$\deg(\xi')\,\equiv\, s\deg(\xi)\ \ \pmod{r}\, ,$$ with 
$s\in \{\pm 1\}$ then $r$ divides $\deg(\xi'\otimes \xi^{-s})$, so there exists $L\in \Pic(X)$ 
such that $L^r=\xi'\otimes \xi^{-s}$. Therefore, $\xi'= \xi^s \otimes L^r =(\xi \otimes (L^s)^r)^s$ and, thus, we 
have an isomorphism $\ST_{\id,L^s,s}:\SM(X,r,\xi) \longrightarrow \SM(X,r,\xi')$.
\end{proof}

\section{Moduli space of framed bundles}
\label{section:framedBundles}

Let $X$ be a smooth complex projective curve. Fix a point $x\,\in\, X$. A framed bundle
on $(X,\,x)$ is a pair $(E,\,\alpha)$ consisting of a vector bundle $E$ over $X$ and a
nonzero $\CC$-linear homomorphism
$$\alpha\,:\,E_x \,\longrightarrow\, \CC^r\, .$$

Given a real number $\tau>0$, we say that a framed bundle $(E,\,\alpha)$ is $\tau$-stable (respectively $\tau$-semistable) if for all proper subbundles $0\,\subsetneq\,
E'\,\subsetneq\, E$
$$\frac{\text{degree}(E')-\epsilon(E',\alpha)\tau}{\rk(E')} < \frac{\text{degree}(E)-\tau}{\rk(E)} \quad 
(\text{respectively, }\le)$$
where
$$\epsilon(E',\alpha)=\left\{ \begin{array}{ll}
1 & \text{if } E_x'\,\not\subseteq\, \ker(\alpha)\\
0 & \text{if }E_x'\,\subseteq\, \ker(\alpha).
\end{array} \right.$$

In the general framework of framed modules introduced in \cite{HL95modules}, a framed 
bundle is a framed module with respect to the reference sheaf $\SO_x^{\oplus r}$. The 
stability condition for framed bundles described here coincides with the stability condition 
defined by Huybrechts and Lehn for framed modules. Fix a line bundle $\xi$
on $X$. Let $\SF\,=\,\SF(X,x,r,\xi,\tau)$ be the moduli space of $\tau$-semistable framed 
bundles $(E,\,\alpha)$ on $(X,\,x)$ with ${\rm rank}(E)\,=\,r\, \geq\, 2$ and $\det(E)
\,=\, \bigwedge^r E\,\cong\, \xi$; it is a complex projective variety \cite{HL95modules}.

Given a fixed rank $r\, \geq\, 2$ and a degree $d$, we say that a stability parameter $\tau$ is generic if there
do not exist integers $d'\,\in \,\mathbb{Z}$ and $0\,<\,r'\,<\,r$ such that
$$rd'-r'd\,=\,-r'\tau \quad \text{or} \quad rd'-r'd\,=\,(r-r')\tau\, .$$
By definition, given a fixed determinant $\xi$ of degree $d$, if $\tau$ is a generic stability
parameter, then there exists no strictly semistable framed bundle in $\SF(X,x,r,\xi,\tau)$. In
particular, by \cite[Theorem 4.1]{HL95modules} and the computation in \cite[Lemma 1.3]{BGM10}
this implies that $\SF(X,x,r,\xi,\tau)$ is smooth. Moreover, by \cite[Theorem 0.1]{HL95modules}, the
moduli space $\SF(X,x,r,\xi,\tau)$ is fine, i.e., it admits a Poincar\'e family.

Let $$\SF^{\op{ss-vb}}\,=\,\SF^{\op{ss-vb}}(X,x,r,\xi,\tau)$$ be the open subset consisting of pairs 
$(E,\alpha)$ such that the vector bundle $E$ is semistable. Similarly, let $$\SF^{\op{s-vb}}
\,=\,\SF^{\op{s-vb}}(X,x,r,\xi,\tau)$$
be the 
subset of pairs $(E,\alpha)$ such that $E$ is stable. Then there is a forgetful map
\begin{eqnarray*}
\xymatrixrowsep{0.05pc}
\xymatrixcolsep{0.3pc}
\xymatrix{
f&:&\SF^{\op{ss-vb}} \ar[rrrr] &&&& \SM\\
&& (E,\alpha) \ar@{|->}[rrrr] &&&& E.
}
\end{eqnarray*}

\begin{proposition}
\label{prop:forgetfulStable}
Let $E$ be a semistable vector bundle, and let $$\alpha:E|_x\longrightarrow \CC^r$$ be an
isomorphism. Then $(E,\,\alpha)$ is $\tau$-stable for every $\tau>0$.
\end{proposition}

\begin{proof}
If $\alpha$ is an isomorphism, then for every subbundle $E'\,\subsetneq\,
E$, we have $\alpha|_{E'}\not\cong 0$, so $\epsilon(E',\alpha)=1$. Now as
$\rk(E')<\rk(E)$ and $\tau>0$, we have
$$\frac{-\epsilon(E',\alpha)\tau}{\rk(E')} = \frac{-\tau}{\rk(E')} < \frac{-\tau}{\rk(E)}$$
so if $E$ is semistable, then
$$\frac{\text{degree}(E')}{\rk(E')}-\frac{-\epsilon(E',\alpha)\tau}{\rk(E')} <
\frac{\text{degree}(E)}{\rk(E)}- \frac{\tau}{\rk(E)}$$
for every $E'\,\subsetneq\, E$.
\end{proof}

\begin{corollary}
\label{cor:forgetfulEquidimensional}
The forgetful map $f:\SF^{\op{ss-vb}}\longrightarrow \SM$ is surjective and its restriction to $f^{-1}(\SM^s)$ has equidimensional fibers.
\end{corollary}

\begin{proof}
Surjectivity is a direct consequence of Proposition \ref{prop:forgetfulStable}. Moreover, for each
$E\in \SM^s$, the preimage $f^{-1}(E)$ can be clearly identified with a subset
of $\PP(\Hom(E|_x,\CC^r))$. As $\tau$-stability is an open condition, it follows that
$$\dim(f^{-1}(E))=\dim(\PP(\Hom(E|_x,\CC^r))) = r^2-1\, .$$
\end{proof}

For every subset $\SU\subset \SM$, let us denote $\SF_\SU=f^{-1}(\SU)$. By Corollary 
\ref{cor:forgetfulEquidimensional} and Proposition \ref{prop:forgetfulStable}, if $\SU\subset 
\SM^s$ and $E\in \SU$, then $f^{-1}(E)$ can be identified with a subset of $\PP(\Hom(E|_x,\CC^r))$ 
that contains $\PP(\op{Iso}(E|_x,\CC^r))$. Let $\SF_\SU^0$ be the subset of $\SF_\SU$ consisting of pairs $(E,\alpha)$ such that $\alpha$ is an isomorphism. Clearly, it is a $\PGL_r(\CC)$-bundle on 
$\SU$.

We shall compute the codimension of the complement of certain relevant open subsets.

\begin{lemma}
\label{lemma:codimUnstable}
Let $k\,>\,0$ be an integer, and let $\tau$ be a generic stability parameter. If $
(r-1)(g-2)+g+1\, \geq\, k$, then the codimension of the closed
subset $\SF\backslash \SF^{\op{s-vb}}$ in $\SF$ is at least $k$. In particular, if $g\,\ge\, 4$,
the space has codimension at least $2$ for every rank.
\end{lemma}

\begin{proof}
Let $(E,\,\alpha)$ be a $\tau$-stable framed bundle such that $E$ is not stable. Let $d=\deg(E)$, $r=\rk(E)$. Then there exists a subbundle $F\subset E$ with rank $\rk(F)\,=\,r'$ and degree $\deg(F)\,=\,d'$
such that
$$\frac{\deg(F)}{\rk(F)} \,\ge\, \frac{\deg(E)}{\rk(E)}$$
or, equivalently
$$d'r-dr' \ge 0$$
in particular, $\alpha\vert_F\,\not=\, 0$, as otherwise $(E,\alpha)$ would be $\tau$-unstable. Therefore,
we have $\epsilon(F,\alpha)\,=\,1$, so, as $\tau$ is generic,
$$\frac{d'-\tau}{r'} < \frac{d-\tau}{r}$$
and we have
$$0\le d'r-dr'<(r-r')\tau\, .$$

From the codimension estimation in \cite[p.~247--248]{Bh} it follows that the locus of all $\tau$-stable
framed bundles such that underlying vector bundle is not semistable has codimension at least $(r-1)(g-2)+g+1$
in $\mathcal F$. The estimation in \cite[p.~247--248]{Bh} is for a Harder--Narasimhan filtration of fixed type,
but there are only finitely many Harder--Narasimhan filtrations that occur in a given bounded family.
Therefore, we may assume that $E$ and $F$ are semistable with
$$\frac{\deg(F)}{\rk(F)} \,=\, \frac{\deg(E)}{\rk(E)}\, .$$
Then $E$ fits in a short exact sequence
$$0\longrightarrow F \longrightarrow E \longrightarrow Q \longrightarrow 0$$
where $Q$ is also semistable. Let $d''=\deg(Q)=d-d'$, $r''=\rk(Q)=r-r'$. Then from the previous discussion
it follows that
$$d'r''-d''r' =0$$
Let us compute the dimension of the space $\SZ$ of all possible vector
bundles $E$ with rank $r$ and determinant $\xi$ that can be generated as
extensions of semistable bundles $F$ and $Q$ with the given rank and degree restrictions.
Given fixed $F$ and $Q$, the space of possible vector bundles resulting from
extensions has dimension $$\dim H^1(X,\, \Hom(Q,\,F))-1\,=\,\dim H^0(X,\, F^\vee\otimes Q\otimes K_X)
-1\, .$$ As $F$ and $Q$ are semistable, so is $F^\vee\otimes Q\otimes K_X$. As $g\ge 2$, then
$$\deg(F^\vee\otimes Q\otimes K)\,=\,r'r''(2g-2)-d'r''+d''r' \,=\, r'r''(2g-2)$$
So we have that for $g\ge 2$
$$0\le \frac{\deg (F^\vee\otimes Q\otimes K)}{\rk(F^\vee\otimes Q\otimes K)} = 2g-2$$
and, from Clifford's inequality \cite[Theorem 2.1]{BPGN97}, we know that 
\begin{multline*}
\dim(h^0(F^\vee\otimes Q \otimes K_X))\le \rk(F^\vee\otimes Q\otimes K_X)+\frac{\deg(F^\vee\otimes Q\otimes K_X)}{2} \\
= r'r''+\frac{r'r''(2g-2)}{2}=r'r''g.
\end{multline*}
We know that $F\in \SM(X,r',d')$ and $Q\in \SM(X,r'',d'')$ and we need to impose the extra
condition that $\deg(E)=\xi$, so
\begin{multline*}
\dim(\SZ)\le \max_{0<r'<r}\left\{\dim(\SM(X,r',d'))+\dim(\SM(X,r'',d''))+r'r''g-1-g\right\}\\
=\max_{0<r'<r}\left\{(r')^2(g-1)+1+(r'')^2(g-1)+1+r'r''g-1-g\right\}.
\end{multline*}
Observe that for each $E\,\in\, \SZ$ the space of possible $\alpha\,:\,E|_x\,\longrightarrow\,
\CC^r$ has dimension $r^2-1$, so
$$\dim(\SF\backslash \SF^{\op{ss-vb}})\,\le\, \dim(\SZ)+r^2-1\, .$$
Therefore,
\begin{multline*}
\dim(\SF)-\dim(\SF\backslash \SF^{\op{ss-vb}})\,\geq\, (r^2-1)g-\dim(\SZ)-r^2+1 \\
\,=\, \min_{0<r'<r}\left\{r'r''(g-2)\right\}\,=\,(r-1)(g-2)\,.
\end{multline*}
The genus condition in the statement of the lemma implies
that $$(r-1) (g-2)\,\ge\, k\, ,$$ so we obtain the desired bound on
codimension.
\end{proof}

\begin{corollary}
\label{cor:codimF_U}
Let $k\,>\,0$ be an integer, and let $\tau$ be a generic stability parameter. Let $\SU\,\subset\, \SM^s$ be
an open subset such that $\op{codim}(\SM^s\backslash \SU,\SM^s)\,\ge\, k$. If
$(r-1)(g-2)+g+1\, \geq\, k$, then
$$\op{codim}(\SF\backslash \SF_\SU,\SF)\ge k\, .$$
In particular, if $g\ge 2$ then for every open subset $\SU\subset \SM^s$ whose complement has codimension at least 2, the complement of $\SF_U$ has codimension at least $2$.
\end{corollary}

\begin{proof}
By Lemma \ref{lemma:codimUnstable}, $\op{codim}(\SF\backslash \SF^{\op{s-vb}},\SF)\,\ge\, k$, so it is enough
to prove that $$\op{codim}(\SF^{\op{s-vb}}\backslash \SF_\SU,\SF^{\op{s-vb}})\,\ge\, k\, .$$ By
Corollary \ref{cor:forgetfulEquidimensional}, the forgetful map is equidimensional, so
$$\op{codim}(\SF^{\op{s-vb}}\backslash \SF_\SU,\SF^{\op{s-vb}})\,=\,
\op{codim}(\SM^s\backslash \SU, \SM^s)\,\ge\, k\, .$$
\end{proof}

\begin{lemma}
\label{lemma:genericSubset}
If $g\,\ge \,\max \left\{1+\frac{\tau+k-1}{r-1},\, 2+\frac{k}{r-1}\right\}$, then there exists an open subset $\SU\subset \SM^s$ such that 
\begin{enumerate}
\item $\op{codim}(\SM\backslash \SU,\SM)\,\ge\, k$

\item $\op{codim}(\SF\backslash \SF_\SU,\SF)\,\ge\, k$

\item For each $E\,\in\, \SU$, $$\op{codim}(\PP(\Hom(E|_x,\CC^r))\backslash f^{-1}(E),
\, \PP(\Hom(E|_x,\CC^r)))\ge 2\, .$$
\end{enumerate}
In particular, for $g\, \ge \, {\rm max}\{2+\tau,\, 4\}$ the considered spaces have
codimension at least two for any rank.
\end{lemma}

\begin{proof}
Property (2) follows from (1) by Corollary \ref{cor:codimF_U}, so it is enough to construct a 
subset that satisfies (1) and (3). Let $$\SZ=\left\{E\in \SM^s | \exists L\subset E \, 
\deg(L)>\frac{\deg(E)-\tau}{\rk E} \right\}\, .$$ Take $E\in \SM^s\backslash \SZ$ and let 
$\alpha\,:\,E|_x\,\longrightarrow\,\CC^r$ be any nonzero map.

If $\alpha$ is an isomorphism, then $(E,\alpha)$ is $\tau$-stable by Proposition \ref{prop:forgetfulStable}, so $(E,\alpha)\in f^{-1}(E)$.

Suppose that $\dim \ker(\alpha)=1$, and let $E'\subset E$ be any subbundle. As
$\dim \ker(\alpha)=1$ and $E'$ is saturated, if $\rk(E')>1$, then we have $\epsilon(E',\alpha)=1$, so
$$\frac{\deg(E')-\epsilon(E',\alpha)\tau}{\rk(E')}=\frac{\deg(E')-\tau}{\rk(E')}<\frac{\deg(E)-\tau}{\rk(E)}$$
because $\tau>0$, $\rk(E')<\rk(E)$ and $\frac{\deg(E')}{\rk(E')} < \frac{\deg(E)}{\rk(E)}$ by stability of $E$. On the other hand, if $\rk(E')=1$, then
$$\frac{\deg(E')-\epsilon(E',\alpha)\tau}{\rk(E')}=\deg(E')-\epsilon(E',\alpha) \tau\le \deg(E')\le \frac{\deg(E)-\tau}{\rk(E)}$$
because $E\not\in \SZ$. Therefore, $(E,\alpha)$ is $\tau$-semistable and, as $\tau$ is generic, $(E,\alpha)\in f^{-1}(E)$.

Therefore, $\{\alpha\in \PP(\Hom(E|_x,\CC^r))\,\mid\, \rk(\alpha)\ge r-1\}\subseteq f^{-1}(E)$.

As $\op{codim}(\{\alpha\in \PP(\Hom(E|_x,\CC^r)) \,\mid\, \rk(\alpha) \le r-2\})$ is $4$ for $r>2$,
and it is $3$ for $r=2$, we conclude that condition (3) holds for any $E\in \SM\backslash \SZ$.

Take $\SU=\SM^s\backslash\SZ$. Let us prove that $\op{codim}(\SM\backslash \SU,\SM)\ge k$. By
\cite[Lemma 2.3]{BGM10}, $\op{codim}(\SM\backslash \SM^s,\SM) \ge (r-1)(g-1)$, so the given bound on genus implies that $\op{codim}(\SM\backslash \SM^s,\SM)\ge k$. Therefore, it is enough to prove that $\op{codim}(\SZ,\SM)\ge k$. Let $E\in \SZ$. Then $E$ fits in a short exact sequence
$$0\longrightarrow L \longrightarrow E \longrightarrow Q \longrightarrow 0$$
where $L$ is a line bundle of degree $d'$ and $Q$ is a rank $r-1$ bundle of degree $d''\,=\,\deg(E)-d'$ such that
$$\frac{\deg(E)-\tau}{\rk(E)}<d'<\frac{\deg(E)}{r} < \frac{d''}{r-1}$$
or, equivalently
$$0<\deg(E)-rd'<\tau$$
Moreover, generically, we can choose $Q$ to be stable.

The space of possible line bundles $L$ has dimension $g$ and the space of choices for $Q$ has dimension $(r-1)^2(g-1)+1$. Given fixed $L$ and $Q$, the space of
vector bundles $E$ fitting in the previous short exact sequence has dimension $\dim H^1(\Hom(Q,\,L))-1$. If we take into account that $E$ has to have fixed determinant $\xi$ in order to belong to $\SM(X,r,\xi)$, then this dimension drops to $\dim(H^1(\Hom(Q,L)))-1-g$. Observe that for stable
$Q$, as $\deg(L)<\frac{\deg(Q)}{\rk(Q)}$ we have $H^0(\Hom(Q,L))=0$, so
\begin{multline*}
\dim(H^1(\Hom(Q,\,L)))\,=\,-\chi(\Hom(Q,\,L))\,=\,d''-(r-1)d'+(r-1)(g-1)\\
=\,\deg(E)-rd'+(r-1)(g-1).
\end{multline*}
Therefore
\begin{multline*}
\dim(\SZ)=\max_{\frac{\deg(E)-\tau}{\rk(E)}<d'<\frac{\deg(E)}{\rk(E)}} \{g+(r-1)^2(g-1)+1+\deg(E)-rd'+(r-1)(g-1)-1-g\}\\
=\max_{\frac{\deg(E)-\tau}{\rk(E)}<d'<\frac{\deg(E)}{\rk(E)}} \{r(r-1)(g-1)+\deg(E)-rd'\}<r(r-1)(g-1)+\tau .
\end{multline*}
So
$$\dim(\SM)-\dim(\SZ)>(r^2-1)(g-1)-r(r-1)(g-1)-\tau=(r-1)(g-1)-\tau\, .$$
The genus condition in the statement of the theorem implies that $(r-1)(g-1)-\tau\ge k-1$, so $\op{codim}(\SZ)\ge k$.
\end{proof}

\section{The $\PGL_r(\CC)$-action and Torelli theorem}
\label{section:Torelli}

Make $\PGL_r(\CC)$ act on $\SF$ by composition with the framing $\alpha$. Given a matrix $[G]\in \PGL_r(\CC)$, where $G\in \GL_r(\CC)$ is any representative of the
projective class, the automorphism $G\,:\,\CC^r \,\longrightarrow\, \CC^r$ produces the
self-map 
$$(E,\,\alpha)\,\longmapsto\, (E,\,G\circ \alpha)$$
of framed bundles. Since for every subbundle $E'\,\subset\, E$ we have
$$\epsilon(E',\,\alpha)\,=\,\epsilon(E',\,G\circ \alpha)$$
this action preserves the (semi)stability condition and it is a well defined
map $\Psi_{[G]}\,:\,\SF\,\longrightarrow \,\SF$, giving rise to an effective action
$\PGL_r(\CC)\times \SF\,\longrightarrow\, \SF$ (see \cite[Lemma 2.6]{BGM10} for more details).
The forgetful map $\SF^{\op{ss-vb}}\,\longrightarrow\, \SM$ is evidently $\PGL_r(\CC)$-invariant.

We shall prove that this action is essentially the unique possible effective action of $\PGL_r(\CC)$ on 
$\SF$ and that the quotient $\SF\gitq \PGL_r(\CC)$ is an extension of the forgetful map. Moreover, 
we shall prove that every isomorphism between moduli spaces of framed bundles $\SF\,\longrightarrow\,\SF'$ can 
be factored as a composition of a map of the form $\Psi_{[G]}$ with a $\PGL_r(\CC)$-equivariant 
isomorphism.

\begin{lemma}
\label{lemma:sectionsTangent}
If $g\,\ge\,{\rm max}\{2+\tau,\, 4\}$, then the $\PGL_r(\CC)$-action on $\SF$
induces an isomorphism $H^0(\SF,\, T_\SF)) \,\cong\, \mathfrak{pgl}_r(\CC)$.
\end{lemma}

\begin{proof}
We will proceed following the ideas of \cite[Lemma 2.7, Lemma 2.8 and Corollary 2.9]{BGM10}. Let 
$\SU$ be the open subset of $\SM^s$ given by Lemma \ref{lemma:genericSubset}. Let $f_\SU\,:\,\SF_\SU\,\longrightarrow\, 
\SU$ be the restriction of the forgetful map, and let $T_{f_\SU}\subset T_{\SF_\SU}$ be the relative 
tangent sheaf for the map $f_\SU$, i.e., it is the kernel of the differential
$$df_\SU\,:\,T_{\SF_\SU} 
\,\longrightarrow\, f_\SU^*T_{\SU}\,.$$ Observe that the forgetful map $f_\SU\,:\,\SF_\SU\,\longrightarrow\, \SM_\SU$ is 
$\PGL_r(\CC)$-invariant with respect to the previous action, so it induces a homomorphism of Lie 
algebras
\begin{equation}\label{ea}
a: \mathfrak{pgl}_r(\CC) \longrightarrow H^0(\SF_\SU,\,T_{f_\SU})\,.
\end{equation}

The action of $\PGL_r(\CC)$ on the set $\SF_\SU^0$ of framed bundles with invertible framing is free, so the composition
$$\mathfrak{pgl}_r(\CC) \stackrel{a}{\longrightarrow} H^0(\SF_\SU,\,T_{f_\SU})
\,\longrightarrow \,H^0(\SF_\SU^0,\,T_{f_\SU}|_{\SF_\SU^0})$$
is injective. Therefore, $a$ is injective.

We will prove that the map $a$ in \eqref{ea} is an isomorphism. For that
it suffices to show that $\dim H^0(\SF_\SU,\,T_{f_\SU})\,=\,\dim \mathfrak{pgl}_r(\CC)\,=\,r^2-1$.

Let $$\pi\, :\, \PP_{\SM^s,x}\,\longrightarrow\, \SM^s$$ be the projective bundle over $\SM$ whose
fiber over any $E$ is $\PP(\Hom(E|_x,\,\CC^r))$, and let $\PP_\SU$ be its restriction to $\SU$. By the choice
of $\SU$, we have a commutative diagram
\begin{eqnarray*}
\xymatrix{
\SF_\SU \ar@{^(->}[r] \ar[rd]_{f_\SU} & \PP_\SU \ar[d]^{\pi_\SU} \ar@{^(->}[r] & \PP_{\SM^s,x} \ar[d]^{\pi}\\
&\SU \ar@{^(->}[r] & \SM^s
}
\end{eqnarray*}
Clearly, the $\PGL_r(\CC)$ action on $\SF$ extends to $\PP_\SU$, and the above map $\pi_\SU$ is
$\PGL_r(\CC)$-invariant, so the map
$$a\,:\,\mathfrak{pgl}_r(\CC) \,\longrightarrow \,H^0(\SF_\SU,\,T_{f_\SU})$$ factors through the map $H^0(\PP_\SU,\, T_{\pi_\SU}) \longrightarrow H^0(\SF_\SU,\,T_{f_\SU})$. As the codimension of the complement of $\SF_\SU$ in $\PP_\SU$ is at least $2$ and $\PP_\SU$ is smooth,
$$H^0(\SF_\SU,\,T_{f_\SU})\,\cong\, H^0(\SF_\SU, \,T_{\pi_\SU})\,=\,H^0(\PP_\SU,\,T_{\pi_\SU})\, ,$$
so it is enough to prove that
\begin{equation}\label{ea2}
\dim(H^0(\PP_\SU,\, T_{\pi_\SU}))\,=\,r^2-1\, .
\end{equation}

Once again, the complement of $\SU$ in $\SM^s$ has codimension at least $2$, and the map $\pi$ is equidimensional, so the complement of $\PP_\SU$ in $\PP_{\SM^s,x}$ has codimension at least $2$. As the latter is smooth, we obtain that
$$H^0(\PP_\SU,\, T_{\pi_\SU})\,\cong\, H^0(\PP_\SU,\, T_{\pi}) \,=\, H^0(\PP_{\SM^s,x},\,T_{\pi})\, .$$
By \cite[Lemma 2.7]{BGM10} we have $\dim H^0(\PP_{\SM^s,x},\,T_{\pi})\,=\,r^2-1$. Since \eqref{ea2} is proved,
we conclude that the map $a$ in \eqref{ea} is an isomorphism.

Since the fibers of the map $f_\SU$ are open subsets of codimension at least $2$ of projective spaces, we have
$$H^0(\SF_\SU,\,f_\SU^*T_{\SU})\,=\,H^0(\SU,\,T_\SU)\, .$$
Moreover, as the complement of $\SU$ in $\SM^s$ has codimension at least $2$ and $\SM^s $ is smooth, we have
$$H^0(\SU,\,T_\SU)\,\cong\, H^0(\SU,\,T_{\SM^s})\,=\,H^0(\SM^s,\,T_{\SM^s})\, .$$
Then, the proof of \cite[Theorem 6.2]{Hi2} implies that $H^0(\SM^s,\,T_{\SM^s})\,=\,0$ (the proof was originally stated in
rank two, but works in any rank. In the coprime case this was first proven in \cite{NR75}). On the other hand, we have the short exact sequence
$$0\,\longrightarrow \,H^0(\SF_\SU,\,T_{f_\SU}) \,\longrightarrow \,H^0(\SF_\SU,\,T_{\SF_\SU})
\,\longrightarrow\, H^0(\SF_\SU,\,f_\SU^*T_\SU)\,=\,0\, ,$$
so $\mathfrak{pgl}_r(\CC)\,\cong\, H^0(\SF_\SU,\,F_{f_\SU})\,=\,H^0(\SF_\SU,\,T_{\SF_\SU})$. Finally, as the complement of $\SF_\SU$ in $\SF$ has codimension at least $2$ and $\SF$ is smooth, we have
$$H^0(\SF_\SU,\,T_{\SF_\SU})\,\cong\, H^0(\SF_\SU,\,T_\SF)\,=\,H^0(\SF,\,T_\SF)\, .$$
This completes the proof of the lemma.
\end{proof}

\begin{proposition}
\label{prop:uniqueAction}
Let $X$ be a curve of genus $g\,\ge\, {\rm max}\{2+\tau,\, 4\}$. Then the above action $\Psi$
of $\PGL_r(\CC)$ on $\SF$
is the unique effective action of $\PGL_r(\CC)$ on $\SF$ up to a group automorphism of $\PGL_r(\CC)$.
\end{proposition}

\begin{proof}
The proof is completely analogous to that of \cite[Proposition 2.5]{BGM10}. Any effective 
$\PGL_r(\CC)$ action induces an injection $i\,:\,\mathfrak{pgl}_r(\CC)\,\hookrightarrow\, H^0(\SF,\,T_\SF)$. 
By Lemma \ref{lemma:sectionsTangent}, $\dim H^0(\SF,\,T_\SF)\,=\,r^2-1\,=\,\dim \mathfrak{pgl}_r(\CC)$, so 
$i$ is an isomorphism. Let $j:\mathfrak{pgl}_r(\CC) \,\cong\, H^0(\SF,\,T_\SF)$ be the isomorphism of Lie 
algebras induced by a second effective action. Then $i\circ j^{-1}$ is an automorphism of 
$\mathfrak{pgl}_r(\CC)$ which comes from an automorphism of the group $\PGL_r(\CC)$.
\end{proof}

\begin{lemma}\label{lez}
Assume that $g\,\ge\,{\rm max}\{2+\tau,\, 4\}$. Then there is a short exact sequence
$$1\longrightarrow \Pic (\SM) \longrightarrow \Pic(\SF) \longrightarrow \ZZ \longrightarrow 1\, ,$$
where the homomorphism $\Pic (\SM) \longrightarrow \Pic(\SF)$ is the extension of pullback of line bundles from $\SM$,
and the second homomorphism is
the restriction to a generic fiber of the forgetful map $f\,:\,\SF^{\op{ss-vb}}\,\longrightarrow\, \SM$.
\end{lemma}

\begin{proof}
Let $\SU$ be the open subset given by Lemma \ref{lemma:genericSubset}. As the codimension of the complement
of $\SF_\SU$ in $\SF$ is at least $2$, we have
$\Pic(\SF_\SU)=\Pic(\SF)$. For each $E\,\in\, \SF_\SU$, the fiber $f^{-1}(E)$ can be identified with a subset of $\PP(\Hom(E|_x,\CC^r))$ whose complement has codimension at least $2$, so we have a short exact sequence
$$1\longrightarrow \Pic(\SU) \longrightarrow \Pic(\SF_\SU) \longrightarrow \ZZ \longrightarrow 1\, .$$
The codimension of the complement of $\SU$ in $\SM$ is at least $2$ and $\SM$ is normal, so
$$\Pic(\SU)\,=\,\Pic(\SM^s)\,=\,\Pic(\SM)\,\cong\, \ZZ\, .$$ As $\SF$ is smooth and the complement
of $\SF_\SU$ in
$\SF$ is of codimension at least two, we have $\Pic(\SF_\SU)\,=\,\Pic(\SF)$, and the pullback map induces a homomorphism
$$\Pic(\SM)\,\longrightarrow \,\Pic(\SF_\SU)\,=\, \Pic(\SF)\, ,$$
so we obtain the desired short exact sequence.
\end{proof}

\begin{lemma}
\label{lemma:GIT}
Assume that $g\,\ge \,{\rm max}\{2+\tau,\, 4\}$. A point $(E,\,\alpha)\,\in \,
\SF^{\op{ss-vb}}$ is $\PGL_r(\CC)$-semistable with respect to
any linearized polarization if and only if $\alpha$ is an isomorphism.
\end{lemma}

\begin{proof}
The proof is completely analogous to that of \cite[Lemma 3.2]{BGM10} after using Lemma
\ref{lez} instead of \cite[Lemma 3.1]{BGM10}.
\end{proof}

\begin{lemma}\label{lemma:isoFibers}
Let $X$ and $X'$ be smooth curves of genus $g$ and $g'$ respectively, and let $\tau$ and $\tau'$
be positive generic
stability parameters such that $g\,\ge\, {\rm max}\{2+\tau,\, 4\}$ and
$g'\,\ge\, {\rm max}\{2+\tau,\, 4\}$. Let $x\in X$ and $x'\,\in \,X'$ be marked points,
and let $\xi$ and $\xi'$ be line bundles over $X$ and $X'$ respectively. Let $\SF\,=\,\SF(X,x,r,\xi,\tau)$ and
$\SF'\,=\,\SF(X',x',r',\xi',\tau')$, and assume that there is an isomorphism
$$\Psi:\SF\stackrel{\sim}{\longrightarrow} \SF'\, .$$ Let $\SM\,=\,\SM(X,r,\xi)$ and $\SM'
\,=\,\SM(X',r',\xi')$. Then $r=r'$, $g=g'$ and there exist open subsets $\SU\subset \SM^s$ and $\SU'\subset (\SM')^s$ such that
\begin{enumerate}
\item $\op{codim}(\SM\backslash \SU,\SM)\ge 2$ and $\op{codim}(\SM'\backslash \SU',\SM')\ge 2$,
\item $\op{codim}(\SF\backslash \SF_\SU,\SF)\ge 2$ and $\op{codim}(\SF'\backslash \SF'_{\SU'},\SF')\ge 2$, and
\item for each $E\,\in\, \SU$ and each $E'\,\in\, \SU'$
\begin{eqnarray*}
\op{codim}(\PP(\Hom(E|_x,\CC^r))\backslash f^{-1}(E),\PP(\Hom(E|_x,\CC^r)))\ge 2\\
\op{codim}(\PP(\Hom(E'|_{x'},\CC^r))\backslash f^{-1}(E'),\PP(\Hom(E'|_{x'},\CC^r)))\ge 2\, ;
\end{eqnarray*}
\end{enumerate}
moreover, there is an isomorphism $\ST:\SU\stackrel{\sim}{\longrightarrow}\SU'$ such that
the following diagram commutes
\begin{eqnarray*}
\xymatrixrowsep{1pc}
\xymatrix{
\SF \ar[r]^{\Psi} & \SF'\\
\SF_{\SU} \ar[r]^{\Psi_{\SU}} \ar@{^(->}[u] \ar[dd]^f & \SF'_{\SU'} \ar@{^(->}[u] \ar[dd]^{f'}\\
&\\
\SU \ar[r]^{\ST} & \SU'
}
\end{eqnarray*}
\end{lemma}

\begin{proof}
By Lemma \ref{lemma:sectionsTangent}
$$r^2-1\,=\,\dim(H^0(\SF,\,T_\SF))\,=\, \dim H^0(\SF',T_{\SF'})\,=\,(r')^2-1\, ,$$
so $r\,=\,r'$. Moreover
$$(r^2-1)g = \dim(\SF) = \dim(\SF')=((r')^2-1)g'=(r^2-1)g'$$
so $g=g'$. Let us fix once and for all a linearized polarization of the $\PGL_r(\CC)$-action on $\SF$. The pullback
of it by $\Psi^{-1}$ gives a linearized polarization on the $\PGL_r(\CC)$-action on $\SF'$. Taking the GIT quotient with respect to those polarizations we obtain a map $\overline{\ST}:\SF\gitq\PGL_r(\CC) \stackrel{\sim}{\longrightarrow} \SF'\gitq\PGL_r(\CC)$ such that the following diagram commutes
\begin{eqnarray*}
\xymatrixrowsep{1pc}
\xymatrix{
\SF \ar[r]^{\Psi} & \SF'\\
\SF^{\op{ss-GIT}} \ar[r]^{\Psi_{\SU}} \ar@{^(->}[u] \ar@{->>}[dd]^\pi & (\SF')^{\op{ss-GIT}} \ar@{^(->}[u] \ar@{->>}[dd]^{\pi'}\\
&\\
\SF\gitq \PGL_r(\CC) \ar[r]^{\overline{\ST}} & \SF'\gitq \PGL_r(\CC)
}
\end{eqnarray*}
Since $f\,:\,\SF^{\op{ss-vb}}\,\longrightarrow\, \SM$ is $\PGL_r(\CC)$-invariant, the restriction to $\SF^*:=\SF^{\op{ss-vb}}\cap
\SF^{\op{ss-GIT}}$ factors as
$$\SF^* \longrightarrow \SF^*\gitq \PGL_r(\CC)\stackrel{\widetilde{g}}{\longrightarrow} \SM\, .$$
Let us prove that $\widetilde{g}$ is an isomorphism. By Lemma \ref{lemma:GIT}, $\SF^*$ coincides with the set of framed bundles $(E,\alpha)\in \SF$ such that $E$ is semistable and $\alpha$ is an isomorphism. The open subset $\SF^0:=\SF^{\op{s-vb}}\cap \SF^{\op{ss-GIT}}$ is a fibration over $\SM^s$ whose fiber over each $E\in \SM^s$ is isomorphic to $\PP(\op{Iso}(E|_x,\CC^r)$ and $\SM$ is normal, so the restriction of $\widetilde{g}:\SF^0\gitq\PGL_r(\CC)\longrightarrow \SM^s$ is an isomorphism. Moreover, the fiber of $\widetilde{g}$ over a strictly semistable bundle is just one point (c.f. \cite[Proposition 3.3]{BGM10}), so $\widetilde{g}$ is an isomorphism. Similarly, we have an isomorphism $(\SF')^*\gitq \PGL_r(\CC) \stackrel{\widetilde{g}'}{\cong} \SM'$ and we have a commutative diagram
\begin{eqnarray*}
\xymatrixrowsep{1pc}
\xymatrix{
\SF^{\op{ss-vb}} \ar@{^(->}[r] &\SF \ar[r]^{\Psi} & \SF' & \ar@{_(->}[l] (\SF')^{\op{ss-vb}}\\
\SF^* \ar@{^(->}[r] \ar[dd]^f \ar@{^(->}[u] & \SF^{\op{ss-GIT}} \ar[r]^{\Psi_{\SU}} \ar@{^(->}[u] \ar@{->>}[dd]^\pi & (\SF')^{\op{ss-GIT}} \ar@{^(->}[u] \ar@{->>}[dd]^{\pi'} & (\SF')^* \ar@{_(->}[l] \ar[dd]^{f'} \ar@{^(->}[u]\\
&&&\\
\SM \ar@{^(->}[r]^-{\widetilde{g}} & \SF\gitq \PGL_r(\CC) \ar[r]^{\overline{\ST}} & \SF'\gitq \PGL_r(\CC)& \SM' \ar@{_(->}[l]_-{\widetilde{g}'}
}
\end{eqnarray*}
Let $\SU_1\,\subset \,\SM^s$ and $\SU_1'\,\subset\,(\SM')^s$ be the subsets given by Lemma \ref{lemma:genericSubset}. Take $\SU=\SU_1\cap (\overline{\ST}^{-1}\circ \widetilde{g}')(\SU_1')$ and $\SU'=\overline{\ST}(\SU)$. By construction, $\Psi(\SF_\SU)=\SF'_{\SU'}$ and $\SU\cong \SU'$, so it is enough to prove that properties (1), (2) and (3) are satisfied.

As $\SU\,\subset\, \SU_1$ and $\SU'\,\subset\, \SU_1'$ we obtain (3). By Corollary \ref{cor:codimF_U}, (2) follows from (1), 
so it is enough to prove that the complements of $\SU$ and $\SU'$ in $\SM$ and $\SM'$ respectively have codimension 
at least $2$.

Let $\mathcal{S}=\SF\gitq \PGL_r(\CC) \backslash \SM$, $\mathcal{S}'=\SF'\gitq \PGL_r(\CC)$, $\SZ_1=\SM\backslash \SU_1$ and $\SZ_1'=\SM\backslash \SU_1'$. The proof of Lemma \ref{lemma:codimUnstable} also implies that under the
given bounds on genus,
$$\dim(\mathcal{S}) \le \dim(\SM)-2 \ \ \text{ and }\ \ \dim(\mathcal{S}')\le \dim(\SM')-2=\dim(\SM)-2\, .$$
Moreover, we chose $\SU_1$ and $\SU_1'$ through Lemma \ref{lemma:genericSubset} so that
$$\dim(\SZ_1)\le \dim(\SM)-2 \ \ \text{ and }\ \ \dim(\SZ_1')\le \dim(\SM')-2=\dim(\SM)-2\, .$$
By construction
$$(\SF \gitq \PGL_r(\CC))\backslash \SU =\mathcal{S} \cup \SZ_1 \cup \overline{\ST}^{-1}(\mathcal{S}' \cup \SZ_1')\, .$$
As $\overline{\ST}$ is an isomorphism, we know that 
$$\dim(\overline{\ST}^{-1}(\mathcal{S}_1'\cup \SZ_1'))=\dim(\mathcal{S}\cup \SZ_1)\le \dim(\SM)-2$$
so $\op{codim}(\SM\backslash \SU,\SM)\ge 2$. Analogously
$$(\SF' \gitq \PGL_r(\CC))\backslash \SU' \,=\,\mathcal{S}' \cup \SZ_1' \cup \overline{\ST}(\mathcal{S} \cup \SZ_1)$$
so we obtain that $\op{codim}(\SM'\backslash \SU',\SM')\ge 2$.
\end{proof}

\begin{corollary}
\label{cor:isoFibers}
Assume that $g\ge \max\{2+\tau,4\}$ and $g'\ge \max\{2+\tau',4\}$, and let $\Psi:\SF \stackrel{\sim}{\longrightarrow} \SF'$ be an
isomorphism. Then $r=r'$, and there is an isomorphism $\sigma:X' \stackrel{\sim}{\longrightarrow} X$, a line bundle
$L$ over $X$ and a sign $s\in\{\pm 1\}$, such the following diagram commutes
\begin{eqnarray*}
\xymatrixrowsep{1pc}
\xymatrix{
\SF \ar[r]^{\Psi} & \SF'\\
\SF^{\op{ss-vb}} \ar[r]^{\Psi^{\op{ss-vb}}} \ar@{^(->}[u] \ar[dd]^f & (\SF')^{\op{ss-vb}} \ar@{^(->}[u] \ar[dd]^{f'}\\
&\\
\SM \ar[r]^{\ST_{\sigma,L,s}} & \SM'
}
\end{eqnarray*}
Moreover, the equality $\Psi^{\op{ss-vb}}(\SF^0)\,=\,(\SF')^0$ holds.
\end{corollary}

\begin{proof}
Let $\SU$ and $\SU'$ be the open subsets of $\SM$ and $\SM'$ given by Lemma \ref{lemma:isoFibers}. By Theorem \ref{theorem:autoVB}, there exists $\sigma:X' \stackrel{\sim}{\longrightarrow} X$, a line bundle $L$ over $X$ and $s\in \{\pm 1\}$ such that the isomorphism $\ST:\SU \stackrel{\sim}{\longrightarrow} \SU'$
satisfies the condition
$$\ST(E)\,\cong \,\ST_{\sigma,L,s}(E)$$
for every $E\in \SU$.
As in the proof of Lemma \ref{lemma:isoFibers}, let us pick any linearization of the
$\SSL_r(\CC)$ action on $\SF$ and the induced linearization on $\SF'$. Let
$$\overline{\ST}\,:\,\SF\gitq \PGL_r(\CC) \,\stackrel{\sim}{\longrightarrow}\, \SF'\gitq \PGL_r(\CC)$$
be the isomorphism induced by $\Psi$. Let us consider the following composition extending $\ST_{\sigma,L,s}$
\begin{eqnarray*}
\xymatrix{
\SM \ar@{^(->}[r]^-{i_\SM} & \SF\gitq\PGL_r(\CC) \ar[r]^{\overline{\ST}} & \SF'\gitq \PGL_r(\CC)\\
\SU \ar@{^(->}[u] \ar[rr]^{\ST_{\sigma,L,s}} &&\SU' \ar@{^(->}[u]
}
\end{eqnarray*}
On the other hand, observe that $\ST_{\sigma,L,s}$ extends to an isomorphism $\SM\longrightarrow \SM'$, so it gives us another map
$$\SM \stackrel{\ST_{\sigma,L,s}}{\longrightarrow}\SM' \stackrel{i_{\SM'}}{\hookrightarrow} \SF'\gitq \PGL_r(\CC)\, .$$
As $\SM$ is irreducible and $\SU$ is dense, there is at most one possible map $\SM\,\longrightarrow\, \SF'\gitq \PGL_r(\CC)$ extending this morphism $\SU\,\longrightarrow\, \SF'\gitq(\PGL_r(\CC))$, so it must coincide with $\overline{\ST}\circ i_\SM$. Therefore, we have a commutative diagram
\begin{eqnarray*}
\xymatrixrowsep{1pc}
\xymatrix{
\SF \ar[r]^{\Psi} & \SF'\\
\SF^{\op{ss-vb}} \ar[r]^{\Psi^{\op{ss-vb}}} \ar@{^(->}[u] \ar[dd]^f & (\SF')^{\op{ss-vb}} \ar@{^(->}[u] \ar[dd]^{f'}\\
&\\
\SM \ar[r]^{\ST_{\sigma,L,s}} & \SM'
}
\end{eqnarray*}
Finally, observe that $\ST_{\sigma,L,s}(\SM^s)=(\SM')^s$, so $\Psi^{\op{ss-vb}}(\SF_{\SM^s})\,=\,\SF'_{(\SM')^s}$.
Moreover, the map $\Psi$ preserves GIT stability, so $\Psi^{\op{ss-vb}}$ must send the set of GIT-semistable
points in $\SF_{\SM^s}$ to the set of GIT-semistable points in $\SF'_{(\SM')^s}$. Therefore,
$$\Psi^{\op{ss-vb}}(\SF^0)\,=\,\Psi^{\op{ss-vb}}(\SF_{\SM^s}\cap \SF^{\op{ss-GIT}}) \,= \,
\SF'_{(\SM')^s}\cap (\SF')^{\op{ss-GIT}}\,=\,(\SF')^0\, .$$
This completes the proof.
\end{proof}

Corollary \ref{cor:isoFibers} shows that we can recover the isomorphism class of the curve from the isomorphism 
class of the moduli space of framed bundles. Let us prove that we can moreover recover the base point $x\,\in \,X$.

Let $\PP\,=\,\PP_{\SM^s,x}$ be the projective bundle over $\SM^s\,=
\,\SM^s(X,r,\xi)$ whose fiber over a stable vector bundle 
$E$ is $\PP(\Hom(E_x,\,\CC^r))$. Even if $\SM^s$ does not admit a universal vector 
bundle, the existence of the bundle $\PP$ is ensured by \cite[Lemma 2.2]{BGM13}. The fiber 
of its dual bundle $\PP^\vee$ over a bundle $E$ is canonically isomorphic to 
$\PP(\Hom(\CC^r,\,E_x))$.

\begin{lemma}
\label{lemma:isoProjective}
Assume that $g\,\ge\, \max\{2+\tau,4\}$ and $g'\ge \max\{2+\tau',\, 4\}$, and let
$\Psi\,:\,\SF \,\stackrel{\sim}{\longrightarrow}\, \SF'$ be an isomorphism. Let $\ST_{\sigma,L,s}:\SM \stackrel{\sim}{\longrightarrow} \SM'$ be the induced isomorphism given by Corollary \ref{cor:isoFibers}. Then $\Psi$ induces an isomorphism $\Psi_\PP:\op{Tot}(\PP_{\SM^s,x}) \stackrel{\sim}{\longrightarrow} \op{Tot}(\PP_{(\SM')^s,x'})$ such that the following diagram commutes
\begin{eqnarray*}
\xymatrixrowsep{1pc}
\xymatrix{
\op{Tot}(\PP_{\SM^s,x}) \ar[r]^-{\Psi_\PP} \ar[dd]^f & \op{Tot}(\PP_{(\SM')^s,x'}) \ar[dd]^{f'}\\
&\\
\SM^s \ar[r]^{\ST_{\sigma,L,s}} & (\SM')^s
}
\end{eqnarray*}
\end{lemma}

\begin{proof}
Let $\SU\,\subset \,\SM^s$ and $\SU'\subset (\SM')^s$ be the open subsets given by Lemma \ref{lemma:isoFibers}. Then $\Psi$ induces an isomorphism $\Psi_\SU:\SF_\SU \stackrel{\sim}{\longrightarrow} \SF_{\SU'}$ over the map $\ST_{\sigma,L,s} : \SU \stackrel{\sim}{\longrightarrow} \SU'$. Moreover, we know that $\SF_\SU$ and $\SF_{\SU'}$ can be identified with subsets of $\PP_{\SM^s}|_\SU$ and $\PP_{(\SM')^s}|_{\SU'}$. Let us prove that the map extends to an isomorphism $\overline{\Psi_\SU}: \PP_{\SM^s}|_\SU \stackrel{\sim}{\longrightarrow} \PP_{(\SM')^s}|_{\SU'}$.

Let $\{\SU_i\}_{i\in I}$ be a covering of $\SU$ by analytic open subsets
such that $\PP_{\SM^s,x}$ is trivial over $\SU_i$ and $\PP_{(\SM')^s,x'}$ is trivial over $\SU_i':=\ST_{\sigma,L,s}(\SU_i)$. Fix trivializations
$$\omega_i\,:\,\PP_{\SM^s,x}|_{\SU_i} \,\longrightarrow\, \SU_i\times \PP^{r^2-1} \ \ \text{ and }\ \
\omega'_i\,:\,\PP_{(\SM')^s,x'}|_{\SU'_i} \,\longrightarrow\, \SU_i' \times \PP^{r^2-1}\, .$$ Then for each $i\in I$ we have an isomorphism
\begin{eqnarray*}
\xymatrixrowsep{1pc}
\xymatrix{
\SU_i\times \PP^{r^2-1} & \SU_i'\times \PP^{r^2-1}\\
\omega_i(\SF_{\SU_i}) \ar[r]^-{\Psi_i} \ar@^{^(->}[u] \ar[dd]^f & \omega_i'(\SF'_{\SU'_i}) \ar@^{^(->}[u] \ar[dd]^{f'}\\
&\\
\SU_i \ar[r]^{\ST_{\sigma,L,s}} & \SU'_i
}
\end{eqnarray*}
For each $E\,\in \,\SU$, the codimension of the complement of $f^{-1}(E)$ in $\PP(\Hom(E|_x,\CC^r))$ is at least $2$, so the complement of each fiber of $\omega_i(\SF_{\SU_i})$ in $\PP^{r^2-1}$ is at least $2$ and, therefore, by Hartogs'
theorem, the composition map
$$\omega_i(\SF_{\SU_i}) \stackrel{\Psi_i}{\longrightarrow} \omega_i'(\SF'_{\SU'_i})
\hookrightarrow \SU_i'\times \PP^{r^2-1} \,\longrightarrow\, \PP^{r^2-1}$$
extends uniquely to a map $\SU_i\times \PP^{r^2-1}\,\longrightarrow\, \PP^{r^2-1}$ and, hence, $\Psi_i$ extends to a morphism
$\overline{\Psi_i}\,:\,\SU_i\times \PP^{r^2-1}\,\longrightarrow\, \SU_i'\times \PP^{r^2-1}$. As the inverse also extends
and their composition is the identity on a dense subset, it follows that $\overline{\Psi_i}$ is an isomorphism. By uniqueness of such extension, the extended maps $\overline{\Psi_i}$ agree on the intersections $\SU_i\cap \SU_j$, so patching them all together they define the desired isomorphism $\overline{\Psi_\SU}: \PP_{\SM^s,x}|_\SU \stackrel{\sim}{\longrightarrow} \PP_{(\SM')^s,x'}|_{\SU'}$.

We have that $\ST_{\sigma,L,s}^*\PP_{\SM^s,x}|_\SU \cong \PP_{(\SM')^s,x'}|_{\SU'}$ and $\SU$ is an open subset
of $\SM^s$ of codimension $2$. As $\SM^s$ is smooth, there is at most one possible extension of
the bundle to $\SM^s$, so $\ST_{\sigma,L,s}^*\PP_{(\SM')^s,x'} \cong \PP_{\SM^s,x}$.
\end{proof}

\begin{theorem}
\label{theorem:Torelli}
Let $X$ and $X'$ be curves of genus $g$ and $g'$ respectively, and let $\tau$ and $\tau'$ be positive generic
stability parameters such that $g\ge \max\{2+\tau,4\}$ and $g'\ge \max\{2+\tau',4\}$. Let $x\in X$ and $x'\in X'$ be marked points, and let
$\xi$ and $\xi'$ be line bundles over $X$ and $X'$ respectively. Let $\SF=\SF(X,x,r,\xi,\tau)$ and
$\SF'=\SF(X',x',r',\xi',\tau')$, and assume that there is an isomorphism
$\Psi:\SF\stackrel{\sim}{\longrightarrow} \SF'$. Then $r=r'$, and there exists an isomorphism
$\sigma:X\stackrel{\sim}{\longrightarrow} X'$ such that $\sigma(x)=x'$.
\end{theorem}

\begin{proof}
By Corollary \ref{cor:isoFibers} there exist an isomorphism $\sigma:X'\stackrel{\sim}{\longrightarrow} X$, a line bundle $L$ over $X$ and $s\in \{\pm 1\}$ such that the following diagram is commutative
\begin{eqnarray*}
\xymatrixrowsep{1pc}
\xymatrix{
\SF \ar[r]^{\Psi} & \SF'\\
\SF^{\op{ss-vb}} \ar[r]^{\Psi^{\op{ss-vb}}} \ar@{^(->}[u] \ar[dd]^f & (\SF')^{\op{ss-vb}} \ar@{^(->}[u] \ar[dd]^{f'}\\
&\\
\SM \ar[r]^{\ST_{\sigma,L,s}} & \SM'
}
\end{eqnarray*}
Moreover, by Lemma \ref{lemma:isoProjective}, $\Psi$ induces an isomorphism $\PP_{\SM^s,x} \cong \ST_{\sigma,L,s}^* \PP_{(\SM')^s,x'}$. By \cite[Corollary 4.2]{BGM10}, this implies that $\sigma(x')=x$.
\end{proof}

\begin{remark}
{\rm Theorem \ref{theorem:Torelli} can be extended to $2$-birational transformations between moduli spaces of
framed bundles. More 
precisely, the result holds if we substitute the isomorphism $\Psi\,:\,\SF\,\longrightarrow\, \SF'$ by an isomorphism 
$\Psi\,:\,\widetilde{\SU}\,\longrightarrow\, \widetilde{\SU}'$ where $\widetilde{\SU}$ and $\widetilde{\SU}'$ are open subsets of $\SF$
and $\SF'$ 
respectively whose complements have codimension at least $2$ in the moduli space.}
\end{remark}

\section{Automorphism group}
\label{section:auto}

Apart from the $\PGL_r(\CC)$-action described in the previous section, we can perform the following transformations on (families of) framed
bundles $(E,\,\alpha)$ which preserve the $\tau$-stability condition:
\begin{enumerate}
\item Given an isomorphism $\sigma:X'\longrightarrow X$ such that $\sigma(x')=x$,
$$(E,\,\alpha)\,\longmapsto \,\left (\sigma^*E, \, \sigma^*\alpha \right)\, .$$
\item Given a line bundle $L$ over $X$, fix a trivialization $\alpha_L:L_x \stackrel{\sim}{\longrightarrow} \CC$. Then send
$$(E,\,\alpha)\,\longmapsto \,\left (E\otimes L,\,\alpha\cdot \alpha_L \right)$$
Since two trivializations $\alpha_L$ and $\alpha_L'$ differ only by a scalar constant, this
map is well defined and furthermore it is independent of the choice of the trivialization $\alpha_L$.
\end{enumerate}
Note that taking the pullback by $\sigma$ and tensoring with $L$ both change the
determinant of the resulting framed bundle. Therefore, in general these transformations do not
induce an automorphism of the moduli space $\SF$, but rather an isomorphism
between two (possibly different) moduli spaces of framed bundles. Given $\sigma$ and $L$ we define the map $\overline{\ST_{\sigma,L,+}}:\SF\longrightarrow \SF'$ as the one that sends
\begin{eqnarray}
\label{e1}
\xymatrixrowsep{0.05pc}
\xymatrixcolsep{0.3pc}
\xymatrix{
\overline{\ST_{\sigma,L,+}} & : & \SF(X,x,r,\xi,\tau) \ar[rrrr] &&&& \SF(X',x',r,\sigma^*(\xi\otimes L),\tau)\\
&& (E,\alpha) \ar@{|->}[rrrr] &&&& \left (\sigma^*(E\otimes L), \sigma^*(\alpha\cdot \alpha_L \right))
}
\end{eqnarray}

We will prove that this type of transformations, together with the $\PGL_r(\CC)$-action generate all possible nontrivial isomorphisms between the moduli spaces of framed bundles.

\begin{lemma}\label{lemma:lemlast}
If $r\,>\,2$, then the two projective bundles $\PP$ and $\PP^\vee$ are not isomorphic.
\end{lemma}

\begin{proof}
We shall break the proof up into several cases because this can be seen from different points
of view.

First assume that $r$ and $\text{degree}(\xi)$ are
coprime. Then there is a Poincar\'e vector bundle over $X\times\SM^s$. Let
$$
W\, \longrightarrow\, \{x\}\times \SM^s \,=\, \SM^s
$$
be the restriction of such a
Poincar\'e bundle to $\{x\}\times \SM^s\, \subset\, X\times\SM^s$. Note that
\begin{equation}\label{e2}
\PP^\vee\,=\, {\mathbb P}(W^{\oplus r})\ \ \text{ and }\ \
\PP\,=\, {\mathbb P}((W^\vee)^{\oplus r})\, .
\end{equation}
Assume that the projective bundles $\PP^\vee$ and $\PP$ are isomorphic. Consequently,
from \eqref{e2} it follows that there
is a line bundle $L_0$ on $\SM^s$ such that
\begin{equation}\label{e3}
(W^\vee)^{\oplus r}\,=\, W^{\oplus r}\otimes L_0\, .
\end{equation}

If $A$ and $B$ are two vector bundles on $\SM^s$ such that $A^{\oplus r}$ is isomorphic
to $B^{\oplus r}$, then $A$ is isomorphic to $B$ \cite[p.~315, Theorem~2]{At}. Therefore,
from \eqref{e3} it follows that $W^\vee$ is isomorphic to $W\otimes L_0$. Hence
the line bundle $\bigwedge^r W^\vee$ is isomorphic to $\bigwedge^r (W\otimes L_0)\,=\,
L^{\otimes r}_0\otimes \bigwedge^r W$. The Picard group of $\SM^s$ is identified with $\mathbb Z$
by sending its ample generator to $1$ \cite{Ra}; let $\ell\, \in\mathbb Z$ be the image of
$\bigwedge^r W$ by this identification of $\text{Pic}(\SM^s)$ with
$\mathbb Z$. We have
\begin{equation}\label{f1}
\text{degree}(\xi)\cdot \ell\,=\, 1 + ar
\end{equation}
for some integer $a$ \cite[p.~75, Remark~2.9]{Ra} (see also \cite[p.~75, Definition~2.10]{Ra}).
Since $\bigwedge^r W^\vee\,=\,L^{\otimes r}_0\otimes \bigwedge^r W$, we also have
\begin{equation}\label{f2}
-\ell\,=\, br+\ell\, ,
\end{equation}
where $b\,\in\, \mathbb Z$ is the image of $L_0$. From \eqref{f1} and \eqref{f2} it
follows that
$$
2\text{degree}(\xi)\cdot \ell\,=\, -\text{degree}(\xi) br\,=\, 2+2ar\, .
$$
This implies that $r\,=\,2$.

Now assume that $r$ and $\text{degree}(\xi)$ have a common factor. Let $$\delta\,=\, \text{g.c.d.}(r,
\, \text{degree}(\xi))\, >\, 1$$ be the greatest common divisor. The Brauer group 
$\text{Br}(\SM^s)$ of $\SM^s$ is the cyclic group ${\mathbb Z}/\delta\mathbb Z$, and
it is generated by the class of the restriction to
$\{x\}\times \SM^s$ of the projectivized Poincar\'e bundle \cite[p.~267, Theorem~1.8]{BBGN}; we shall
denote this generator of $\text{Br}(\SM^s)$ by $\varphi_0$. Now, the class of
$\PP^\vee$ is $\varphi_0$ (tensoring by a vector bundle does not change the Brauer class),
and hence the class of $\PP$ is $-\varphi_0$. If $\PP^\vee$ is isomorphic to $\PP$, then
we have $\varphi_0\,=\, -\varphi_0$, hence $\delta\,=\, 2$ (as it is the order of
$\varphi_0$).

We now assume that $\delta\,=\, 2$. For a suitable ${\mathbb P}^{r-1}_{\mathbb C}$ embedded in $\SM^s$, the restriction of
$\PP^\vee$ to it is the projectivization of the vector
bundle ${\mathcal O}_{{\mathbb P}^{r-1}_{\mathbb C}}\oplus
\Omega^1_{{\mathbb P}^{r-1}_{\mathbb C}}$ \cite[p.~464, Lemma~3.1]{BBN09},
\cite[p.~464, (3.4)]{BBN09}; note that any extension of $\Omega^1_{{\mathbb P}^{r-1}_{\mathbb C}}$
by ${\mathcal O}_{{\mathbb P}^{r-1}_{\mathbb C}}$ splits because
$H^1({\mathbb P}^{r-1}_{\mathbb C},\, T{\mathbb P}^{r-1}_{\mathbb C})\,=\, 0$.
Therefore, if $\PP$ and $\PP^\vee$ are isomorphic,
restricting an isomorphism to this embedded ${\mathbb P}^{r-1}_{\mathbb C}$ it follows
that ${\mathcal O}_{{\mathbb P}^{r-1}_{\mathbb C}}\oplus
\Omega^1_{{\mathbb P}^{r-1}_{\mathbb C}}$ is isomorphic
to $({\mathcal O}_{{\mathbb P}^{r-1}_{\mathbb C}}\oplus T{\mathbb P}^{r-1}_{\mathbb C})
\otimes L'$ for some line bundle $L'$ on ${\mathbb P}^{r-1}_{\mathbb C}$. Since
$T{\mathbb P}^{r-1}_{\mathbb C}$ is indecomposable, in fact it is stable,
from \cite[p.~315, Theorem~2]{At} it follows that $T{\mathbb P}^{r-1}_{\mathbb C}
\otimes L'$ is isomorphic to either ${\mathcal O}_{{\mathbb P}^{r-1}_{\mathbb C}}$ or
$\Omega^1_{{\mathbb P}^{r-1}_{\mathbb C}}$. If
$T{\mathbb P}^{r-1}_{\mathbb C}
\otimes L'$ is isomorphic to ${\mathcal O}_{{\mathbb P}^{r-1}_{\mathbb C}}$,
then we have $r\,=\, 2$. If $T{\mathbb P}^{r-1}_{\mathbb C}
\otimes L'$ is isomorphic to $\Omega^1_{{\mathbb P}^{r-1}_{\mathbb C}}$, we have
$$
r+ (r-1)\cdot \text{degree}(L')\,=\, -r \, ,
$$
so we obtain
$$
-(r-1)\cdot \text{degree}(L')\,=\, 2r \, .
$$
Then we conclude that $r-1$ divides $2$, which implies that either $r\,=\, 2$ or $r\,=\, 3$. However, $r$ is even because $\delta \,=\, 2$, so $r\,=\, 2$.
\end{proof}

\begin{lemma}
\label{lemma:noDual}
Under the conditions of the Lemma \ref{lemma:isoFibers}, there exists an isomorphism $$\sigma\,:\,X'\,
\stackrel{\sim}{\longrightarrow}\, X$$ and a line bundle $L$ over $X$ such that 
$$\ST(E)\cong \ST_{\sigma,L,+}(E)$$
for every $E\in \SM$.
\end{lemma}

\begin{proof}
For $r\,=\,2$, this is a direct consequence of Lemma \ref{lemma:trivialrk2}.

Assume that $r\,>\,2$. To prove by contradiction,
suppose that there exist $\sigma: X'\stackrel{\sim}{\longrightarrow} X$ and $L$ such that the induced isomorphism between $\SM$ and $\SM'$ is $\ST_{\sigma,L,-}$. Let $L'=\sigma^* L$. Then clearly $\ST_{\sigma,L,-}^{-1}=\ST_{\sigma^{-1},L',-}$. Fix a trivialization $\alpha_{L}:L_x \stackrel{\sim}{\longrightarrow} \CC$ and consider the map
\begin{eqnarray*}
\xymatrixrowsep{0.05pc}
\xymatrixcolsep{0.3pc}
\xymatrix{
{\wt{\ST_{\sigma^{-1},L',-}}}&:&\op{Tot}(\PP_{(\SM')^s}) \ar[rrrr]^{\sim} &&&& \op{Tot}(\PP^\vee_{\SM^s})\\
&& (E,\alpha) \ar@{|->}[rrrr] &&&& \left( (\sigma^{-1})^*(E\otimes L')^\vee, (\sigma^{-1})^*(\alpha^t \otimes \alpha_L^t) \right).
}
\end{eqnarray*}
The following diagram is commutative by construction
\begin{eqnarray*}
\xymatrixcolsep{4pc}
\xymatrix{
\op{Tot}(\PP|_{(\SM')^s}) \ar[r]^{\wt{\ST_{\sigma^{-1},L',-}}} \ar[d] & \op{Tot}(\PP^\vee|_{\SM^s}) \ar[d] \\
(\SM')^s \ar[r]^{\ST_{\sigma^{-1},L',-}} & \SM^s
}
\end{eqnarray*}
On the other hand, by Lemma \ref{lemma:isoProjective}, there exists an isomorphism 
\begin{eqnarray*}
\xymatrixrowsep{1pc}
\xymatrix{
\op{Tot}(\PP_{\SM^s}) \ar[r]^-{\Psi_\PP} \ar[dd]^f & \op{Tot}(\PP_{(\SM')^s}) \ar[dd]^{f'}\\
&\\
\SM^s \ar[r]^{\ST_{\sigma,L,-}} & (\SM')^s
}
\end{eqnarray*}
Therefore, composing both we obtain an isomorphism
$${\wt{\ST_{\sigma^{-1},L',-}}}\circ \Psi_\PP:
\op{Tot}(\PP_{\SM^s}) \stackrel{\sim}{\longrightarrow} \op{Tot}(\PP^\vee_{\SM^s})$$ commuting with the
respective projections to $\SM^s$, thus contradicting Lemma \ref{lemma:lemlast}. This completes the proof.
\end{proof}

\begin{lemma}
\label{lemma:extendinversetranspose}
Take $r\,>\,2$, and consider the algebraic automorphism
\begin{eqnarray*}
\xymatrixrowsep{0.05pc}
\xymatrixcolsep{0.3pc}
\xymatrix{
\SD&:&\PGL_r(\CC) \ar[rrrr] &&&& \PGL_r(\CC)\\
&& [G] \ar@{|->}[rrrr] &&&& [(G^{-1})^t].
}
\end{eqnarray*}
Then there does not exist any algebraic automorphism
$$\overline{\SD}\,:\,\PP(\Mat_r(\CC))\,\longrightarrow \,\PP(\Mat_r(\CC))$$ extending $\SD$.
\end{lemma}

\begin{proof}
As $\PGL_r(\CC)$ is dense in $\PP(\Mat_r(\CC))$ and the latter is irreducible, there exists at most one extension of $\SD$ to $\PP(\Mat_r(\CC))$. Let $\SU\,\subsetneq\,
\PP(\Mat_r(\CC))$ be the open subset corresponding to matrices with at least an $(r-1)\times (r-1)$ minor with nonzero determinant. Let $\op{cof}$ be the morphism that sends each matrix $[G]\in \SU$ to its cofactor matrix
$$\op{cof}(G)\,=\,\wedge^{r-1}(G)\, .$$
The entries of the cofactor matrix are determinants of minors of $G$, so they are given
by homogeneous polynomials of degree $r-1$ in the entries of $G$ and, therefore, $\op{cof}$ induces an algebraic map
$$\op{cof}:\SU \longrightarrow \PP(\Mat_r(\CC))\, .$$
Given an invertible matrix $[G]\in \PGL_r(\CC)$, we have that
$$(G^{-1})^t \,=\, \frac{1}{\det(G)}\op{cof}(G)\, .$$
Therefore, $[(G^{-1})^t]=[\op{cof}(G)]$ for every $[G]\in \PGL_r(\CC)$ and $\op{cof}$ is the unique possible extension of $\SD$ to $\SU$. Nevertheless, for $r>2$ this map is not injective. For example, for every $\lambda\in \CC$, let
$$G_\lambda=\left(\begin{array}{cc|c|c}
1 & 0 & 0 & 0\\
\lambda & 1 & 0 & 0\\
\hline
0 & 0 & \id_{r-3} & 0\\
\hline
0 & 0 & 0 & 0
\end{array}\right)$$
Clearly, if $\lambda_1\ne \lambda_2$ then $[G_{\lambda_1}]\ne [G_{\lambda_2}]$ in $\PP(\Mat_r(\CC))$. However, for every $\lambda\in \CC$
$$\op{cof}(G_\lambda)=\left( \begin{array}{c|c}
0_{r-1} & 0\\
\hline
0 & 1
\end{array} \right)$$
So, in particular, $[G_\lambda]\in \SU$ for every $\lambda\in \CC$, which proves that $\SD$ cannot be extended to an injective map on $\SU$.
\end{proof}

\begin{lemma}
\label{lemma:autoEquivariant}
Let $\Psi\,:\,\SF\,\longrightarrow\, \SF'$ be an isomorphism of moduli spaces of framed bundles. Then there exists $[G]
\,\in\, \PGL_r(\CC)$ such that $\Psi_{[G]}\circ \Psi$ is a $\PGL_r(\CC)$-equivariant isomorphism, where $\Psi_{[G]} \, :\, \SF'\, \longrightarrow \, \SF'$ is the automorphism of $\SF'$ induced by composing the framing with the $\PGL_r(\CC)$ action on $\CC^r$.
\end{lemma}

\begin{proof}
Let $\gamma:\PGL_r(\CC) \times \SF\longrightarrow \SF$ and $\gamma'\,:\,\PGL_r(\CC) \times \SF'\,\longrightarrow\, \SF'$ be the natural actions of $\PGL_r(\CC)$ on
$\SF$ and $\SF'$ respectively described before. If $\Psi$ is an isomorphism, it induces another action
$$\gamma'':\PGL_r(\CC)\times \SF\longrightarrow \SF$$ given by
$$\gamma''([X],(E,\alpha))\,=\,\Psi^{-1}(\gamma'([X],\Psi(E,\alpha)))\, .$$
By Proposition \ref{prop:uniqueAction}, there exists a unique action of $\PGL_r(\CC)$ on $\SF$ up to a group automorphism of $\PGL_r(\CC)$. For $r=2$, all the automorphisms of $\PGL_2(\CC)$ are
inner and for $r>2$, the only outer automorphism of $\PGL_r(\CC)$ is the inverse-transpose, i.e., the map
$[X] \longmapsto [(X^{-1})^t]$. Therefore, there exists a matrix $[G]\in\PGL_r(\CC)$ such that either
$$\gamma([X],(E,\alpha))\,=\,\gamma''([G^{-1}XG],\,(E,\,\alpha))
\,=\,(\Psi_{[G]}\circ \Psi)^{-1}(\gamma'([X],\,(\Psi_{[G]}\circ \Psi)(E,\,\alpha)))$$
or 
\begin{multline*}
\gamma([X],\,(E,\,\alpha))\,=\,\gamma''([G^{-1}(X^{-1})^tG],\,(E,\,\alpha))\\
=\,(\Psi_{[G]}\circ \Psi)^{-1}(\gamma'([(X^{-1})^t],\,(\Psi_{[G]}\circ \Psi)(E,\,\alpha)))
\end{multline*}
and it is only necessary to consider the latter when $r>2$. In the first case, as
$\Psi_{[G]}$ is an automorphism of $\SF'$, it
follows that $\Psi_{[G]}\circ \Psi$ is a $\PGL_r(\CC)$-equivariant isomorphism. Let
us prove that the second case is impossible if $r > 2$. Let $\SU$ and $\SU'$ be the open subsets given by Lemma \ref{lemma:isoFibers} and let $\ST:\SU \longrightarrow \SU'$ be the isomorphism induced by $\Psi$.
Take $E\in \SU$, and let $E'=\ST(E)$. Then $\Psi_{[G]}\circ \Psi$ induces an algebraic isomorphism
$$(\Psi_{[G]}\circ \Psi)|_{f^{-1}(E)} \,:\, f^{-1}(E)\,\longrightarrow\, f^{-1}(E')\, .$$
By construction of $\SU$ and $\SU'$ we know that
\begin{eqnarray*}
\op{codim}(\PP(\Hom(E|_x,\CC^r))\backslash f^{-1}(E), \PP(\Hom(E|_x,\CC^r)))\ge 2\\
\op{codim}(\PP(\Hom(E'|_{x'},\CC^r))\backslash f^{-1}(E'), \PP(\Hom(E'|_{x'},\CC^r)))\ge 2
\end{eqnarray*}
so, by and Hartogs' theorem, the map $(\Psi_{[G]}\circ \Psi)|_{f^{-1}(E)}$ extends uniquely to an isomorphism
$$\overline{(\Psi_{[G]}\circ \Psi)|_{f^{-1}(E)}} \,:\, \PP(\Hom(E|_x,\CC^r))
\,\longrightarrow\, \PP(\Hom(E'|_{x'},\CC^r))\, .$$
Fix any trivialization $\alpha:E|_x \stackrel{\sim}{\longrightarrow} \CC^r$
of $E|_x$. As $\alpha$ is an isomorphism, we have $(E,\alpha)\,\in\, f^{-1}(E)$. Let $\alpha'\,=\,(\Psi_{[G]}\circ \Psi)|_{f^{-1}(E)}(\alpha)$. By Lemma
\ref{lemma:GIT}, the composition $\Psi_{[G]}\circ \Psi$ sends $\SF^0$ to $(\SF')^0$, so
$\alpha'$ is an isomorphism. Using the trivializations $\alpha$ and $\alpha'$, we get isomorphisms
$$\PP(\Hom(E|_x,\CC^r)) \,\stackrel{\alpha}{\cong}\, \PP(\Mat_r(\CC))
\,\stackrel{\alpha'}{\cong}\, \PP(\Hom(E'|_{x'},\CC^r))\, ;$$
thus $\overline{(\Psi_{[G]}\circ \Psi)|_{f^{-1}(E)}}$ induces an algebraic isomorphism
$$\widetilde{\Psi}\,:\,\PP(\Mat_r(\CC)) \,\longrightarrow\, \PP(\Mat_r(\CC))\, .$$
Moreover, for every $[X]\,\in\, \PGL_r(\CC)$ we have
$$(\Psi_{[G]}\circ \Psi)|_{f^{-1}(E)}(X\circ \alpha)
\,=\, (X^{-1})^t\circ \alpha'\, ,$$ so $\widetilde{\Psi}([X])\,=\,[X^{-1}]^t$
for every $[X]\in \PGL_r(\CC)$, and therefore, $\widetilde{\Psi}$ extends the inverse-transpose map to an
automorphism of $\PP(\Mat_r(\CC))$, thus contradicting Lemma \ref{lemma:extendinversetranspose}.
\end{proof}

\begin{lemma}\label{lel}
Let $\Psi^0\,:\,\SF^0\,\longrightarrow\, \SF^0$ be a $\PGL_r(\CC)$-equivariant
automorphism of $\SF^0$ commuting with the forgetful map $f^0\,:\,\SF^0 \,\longrightarrow
\,\SM^s$. Then $\Psi^0$ is the identity map.
\end{lemma}

\begin{proof}
If $\Psi^0$ is $\PGL_r(\CC)$-equivariant, then it is an automorphism of $\SF^0$
considered as a $\PGL_r(\CC)$-principal bundle. Let $\SP$ be the universal projective bundle
over $\SM^s$, i.e., the unique projective bundle over $X\times \SM^s$ whose fiber over each
stable vector bundle $E$ is $\PP(E)$. Let $\{U_\alpha\}$ be a trivializing cover of $\SM^s$ for $\SP|_x$, and let $g_{\alpha\beta}:U_\alpha\cap U_\beta \longrightarrow \PGL_r(\CC)$ be the corresponding transition functions. Observe that $\{U_\alpha\}$ is also a trivializing cover for $\PP$ and, thus, for the $\PGL_r(\CC)$-bundle $\SF^0$. It is straightforward to check that the transition functions for $\SF^0$ as $\PGL_r(\CC)$-bundle are $(g_{\alpha\beta}^{-1})^t$. Therefore, we conclude that $\SF^0$ is the $\PGL_r(\CC)$-principal bundle associated to the dual bundle of $\SP|_x$, i.e., $\SP^\vee|_x$. By \cite{BBN09}, the projective bundle
$\SP|_x$ is stable and, therefore, its dual $\SP^\vee|_x$ must also be stable. Applying the
results from \cite{BG08} we know that
 $\SP^\vee|_x$ is simple and, therefore, $\SF^0$ has no nontrivial automorphism, so $\Psi^0$ must be the identity map.
\end{proof}

\begin{lemma}
\label{lemma:uniqueExtension}
Let $M$ be an irreducible smooth complex scheme, and let $U\subset M$ be an open subset whose complement has
codimension at least $2$. Let $\tau$ be a generic stability parameter, and let $(\SE_1,\,\alpha_1)$ and
$(\SE_2,\,\alpha_2)$ be families of framed bundles over $(X,x)$ parametrized by $M$, where $\SE_i$ are vector
bundles over $X\times M$ and $\alpha_i\,:\,\SE|_{\{x\}\times M}\,\longrightarrow\, \SO_M^r$.

Assume that for each $t\in U$, the fiber $(\SE_{i,t},\alpha_{i,t})$ over $t\in U$ is a $\tau$-stable
framed bundle of rank $r$ such that $\det(\SE_{i,t})\cong \xi$. In particular, they define maps
$$f_1,\, f_2\,:\,U\,\longrightarrow\,\SF(X,x,r,\xi,\tau)\, .$$ If $f_1\,=\,f_2$, then
$(\SE_1,\,\alpha_1)\,\cong\, (\SE_2,\,\alpha_2)$.
\end{lemma}

\begin{proof}
By definition of the moduli functor, as $f_1|_U=f_2|_U$, there is an isomorphism of families of framed bundles
$(\SE_1,\alpha_1)|_{X\times U} \stackrel{\rho}\cong (\SE_2,\alpha_2)|_{X\times U}$. In particular,
$\SE_1|_{X\times U}\stackrel{\rho}{\cong}\SE_2|_{X\times U}$. Moreover, we can choose $\rho$ so that $\alpha_2|_U\circ \rho=\alpha_1|_U$. The codimension of the complement of $U$ in $M$ is at least $2$, so
$$\op{codim}(X\times M \backslash X\times U, X\times M)\ge 2\, .$$
As $X\times M$ is smooth, $\rho$ extends to an isomorphism $\SE_1\stackrel{\overline{\rho}}{\cong} \SE_2$. Under this
identification, $\alpha_1,\alpha_2\circ \overline{\rho}\,\in\, H^0(\Hom(\SE_1|_{\{x\}\times M},\,\SO_M^r))$ are
sections which extend the map
$$\alpha_1|_U\,=\,\alpha_2|_U\circ \rho\,\in \,H^0(\Hom(\SE_1|_{\{x\}\times U},\,\SO_U^r))\, .$$ As
both sections coincide over an open dense subset, it follows that $\alpha_1=\alpha_2 \circ \overline{\rho}$, so
we have $(\SE_1,\alpha_1)\cong (\SE_2,\alpha_2)$.
\end{proof}

\begin{theorem}
\label{thm:mainthm}
Let $X$ and $X'$ be curves of genus $g$ and $g'$ respectively, and let $\tau$ and $\tau'$ be positive generic
stability parameters such that $g\ge \max\{2+\tau,4\}$ and $g'\ge \max\{2+\tau',4\}$. Let $x\in X$ and $x'\in X'$ be marked points, and
let $\xi$ and $\xi'$ be line bundles over $X$ and $X'$ respectively. Let $\SF=\SF(X,x,r,\xi,\tau)$ and
$\SF'=\SF(X',x',r',\xi',\tau')$, and assume that there is an isomorphism $\Psi:\SF\stackrel{\sim}{\longrightarrow}
\SF'$. Then $r=r'$ and there exist
\begin{itemize}
\item an isomorphism $\sigma:X' \stackrel{\sim}{\longrightarrow} X$ with $\sigma(x')\,=\, x$,
\item a degree zero line bundle $L\,\in\, J(X)$ with $\sigma^*(\xi\otimes L^{\otimes r})\,
\cong \,\xi'$, and
\item a matrix $[G]\in \PGL_r(\CC)$,
\end{itemize}
such that if we pick any trivialization $\alpha_L:L_x\stackrel{\sim}{\longrightarrow} \CC$ then
$$\Psi(E,\,\alpha)\,=\, \left(\sigma^*(E\otimes L),\,G\circ \sigma^*(\alpha\cdot \alpha_L) \right)$$
for every $(E,\,\alpha)\,\in\, \SF$. In particular, stability parameters $\tau$ and $\tau'$ satisfy that a framed
bundle is $\tau$-stable if and only if it is $\tau'$-stable (i.e., they belong to the same stability chamber).
\end{theorem}

\begin{proof}
By Lemma \ref{lemma:autoEquivariant}, there exists some $[G]\,\in\, \PGL_r(\CC)$ such that $\Psi'=\Psi_{[G^{-1}]}\circ \Psi$ is a $\PGL_r(\CC)$-equivariant
isomorphism. Applying Corollary \ref{cor:isoFibers} and Theorem \ref{theorem:Torelli}, there must exist
an isomorphism $\sigma:X' \longrightarrow X$ with $\sigma(x')\,=\, x$, and a line bundle $L$ over
$X$ and $s\in \{\pm 1\}$ with $\sigma^*(\xi\otimes L^{\otimes r})^s\,
\cong \,\xi'$, such that the following diagram is commutative
\begin{eqnarray*}
\xymatrix{
\SF^{\op{ss-vb}} \ar[r]^-{\Psi'} \ar[d]_f & (\SF')^{\op{ss-vb}} \ar[d]^{f'}\\
\SM \ar[r]^{\ST_{\sigma,L,s}} & \SM' 
}
\end{eqnarray*}
moreover, by Lemma \ref{lemma:noDual}, we know that we can choose $L$ so that $s=1$. Composing with $\overline{\ST_{\sigma,L,+}}^{-1}\,=\,
\overline{\ST_{\sigma^{-1},\sigma^*L^{-1},+}}$, we obtain a map
$$\Psi''\,=\,\overline{\ST_{\sigma,L,+}}^{-1}\circ \Psi':\SF^{\op{ss-vb}} \longrightarrow \SF^{\op{ss-vb}}$$ commuting with the projection to $\SM$. The map $\overline{\ST_{\sigma,L,+}}$ is $\PGL_r(\CC)$-equivariant by construction, so $\Psi''$ is a $\PGL_r(\CC)$-equivariant automorphism of $\SF^{\op{ss-vb}}$ commuting with the projection to $\SM$. By the second part of Corollary \ref{cor:isoFibers},
the automorphism $\Psi''$ preserves $\SF^0$, so it induces a $\PGL_r(\CC)$-bundle map
\begin{eqnarray*}
\xymatrix{
\SF^{0} \ar[rr]^{\Psi^{0}} \ar[dr]_f && \SF^{0} \ar[dl]^f\\
 & \SM^s &
}
\end{eqnarray*}
Using Lemma \ref{lel} we obtain that $\Psi^0$ is the identity map on $\SF^0$. There exists at 
most one extension of $\Psi^{0}$ to $\SF^{\op{ss-vb}}$, because $\SF^0$ is dense in $\SF^{\op{ss-vb}}$ and the latter 
is irreducible. Since the identity map of $\SF^{\op{ss-vb}}$ is one such extension, it follows that
$\Psi''|_{\SF^{\op{ss-vb}}}\,=\,\id_{\SF^{\op{ss-vb}}}$, so we have $\Psi|_{\SF^{\op{ss-vb}}}\,=\,\Psi_{[G]}\circ \overline{\ST_{\sigma,L,+}}$.

So far we have proved that the restriction of $\Psi$ to $\SF^{\op{ss-vb}}$ coincides with $\Psi_{[G]}\circ
\overline{\ST_{\sigma,L,+}}$. Let us prove that
$\Psi(E,\alpha)=\Psi_{[G]}\circ \overline{\ST_{\sigma,L,+}}(E,\alpha)$ for all $(E,\alpha)\in \SF$. As $\SF'$ is a fine moduli space, it admits a universal framed bundle $(\SE',\alpha')$ over $(X',x')$ parametrized by $\SF'$. Taking the pullback by the map $\Psi\,:\,\SF\,\longrightarrow\, \SF'$, we obtain a family $(\SE_\Psi,\,\alpha_\Psi)
\,=\,\Psi^*(\SE',\,\alpha')$ of $\tau$-stable framed bundles over $(X',x')$ parametrized by $\SF$. On the other hand, let $(\SE,\alpha)$ be the universal framed bundle over $\SF$. Then
$$\left(\sigma^*(\SE\otimes p_X^*L),G\circ \sigma^*(\alpha\cdot \alpha_L)\right)$$
is a family of framed bundles over $(X',x')$ parametrized by $\SF$. Moreover, by construction, we know that the fiber over each point in $\SF^{\op{ss-vb}}$ is $\tau'$-stable and, therefore, it induces the map $\Psi_{[G]}\circ \overline{\ST}_{\sigma,L,+}\,:\,\SF^{\op{ss-vb}}\,\longrightarrow\, \SF'$. We know that $\Psi$ coincides with $\Psi_{[G]}\circ \overline{\ST_{\sigma,L,+}}$ over $\SF^{\op{ss-vb}}$. By Lemma \ref{lemma:codimUnstable}, the complement of $\SF^{\op{ss-vb}}$
in $\SF$ is of codimension at least $2$ and we know that $\SF$ is smooth. Therefore, using Lemma
\ref{lemma:uniqueExtension}, we conclude that
$$\Psi^*(\SE',\alpha')\cong \left(\sigma^*(\SE\otimes p_X^*L),G\circ \sigma^*(\alpha\cdot \alpha_L)\right)\, .$$
As a consequence, for each $(E,\alpha)\in \SF$,
$$\Psi(E,\,\alpha)\,=\, \left(\sigma^*(E\otimes L),\,G\circ \sigma^*(\alpha\cdot \alpha_L) \right)\, .$$
Finally, observe that by hypothesis the left hand side of the equality is $\tau'$-stable for every $(E,\alpha)\in \SF$, whereas the right hand side is $\tau$-stable by construction. If $\Psi$ is an isomorphism, then both sides of the equality run respectively over the entire space of $\tau'$-stable and $\tau$-stable framed bundles over $(X',x')$ with rank $r$ and determinant $\xi$. Therefore, the equality implies that a framed bundle is $\tau$-stable if and only if it is $\tau'$-stable.
\end{proof}

Let $J(X)[r]$ denote the $r$-torsion points in the Jacobian of $X$, and let $\Aut(X,x)$ be the 
group of automorphisms of $X$ that fix the point $x\in X$, i.e., $$\Aut(X,x)\,=\,\{\sigma\in \Aut(X) 
\,\mid\, \sigma(x)\,=\,x\}\,.$$

\begin{corollary}
\label{cor:maincor}
Suppose that $g\ge \max\{2+\tau,4\}$ and that $\tau$ is generic. Then the automorphism group of $\SF$ is
$$\Aut(\SF)\cong \PGL_r(\CC)\times \ST $$
for a group $\ST$ fitting in the short exact sequence
$$1\longrightarrow J(X)[r] \longrightarrow \ST \longrightarrow \Aut(X,x) \longrightarrow 1\, .$$
\end{corollary}

\begin{proof}
We proved that the automorphism group is generated by the maps
\begin{itemize}
\item $\Psi_{[G]}$ for each $[G]\in \PGL_r(\CC)$, and
\item $\overline{\ST_{\sigma,L,+}}$ for each $\sigma\,\in\, \Aut(X,x)$ and each $L\in J(X)$ such that $\sigma^*(\xi \otimes L^{\otimes r})\cong \xi$.
\end{itemize}
First of all, the action of $\PGL_r(\CC)$ is faithful and commutes with all of the maps $\overline{\ST_{\sigma,L,+}}$, so we can split the group $\Aut(\SF)$ as a product
$$\Aut(\SF)\cong \PGL_r(\CC) \times \left \langle \left \{ \overline{\ST_{\sigma,L,+}} \, \middle |\, \sigma^*(\xi\otimes L^{\otimes r})\cong \xi \right \} \right \rangle\, .$$
Observe that, by construction, $\overline{\ST_{\sigma,L,+}}$ lies over the automorphism
$$\ST_{\sigma,L,+}\,:\,\SM\,\longrightarrow\, \SM$$ through the forgetful map $f\,:\,\SF
\,\longrightarrow\, \SM$. Since the latter is not trivial for any $\sigma\,\in\, \Aut(X,x)$ and
$L\,\in\, \Pic(X)$, apart from $(\sigma,\,L)\,=\,(\id,\,\SO_X)$,
it follows that $\ST_{\sigma,L,+}\,\ne\, \id$ for $(\sigma,\,L)\,\ne\, (\id,\,\SO_X)$.

Therefore, it suffices to prove that the group
$$\ST\,=\,\left \langle \left \{ \overline{\ST_{\sigma,L,+}} \, \middle |\, \sigma^*(\xi\otimes L^{\otimes r})\cong \xi \right \} \right \rangle$$ consisting of
the maps $\overline{\ST_{\sigma,L,+}}$ which preserve the determinant is an extension of $\Aut(X,x)$ by $J(X)[r]$.

Let $\sigma\in \Aut(X,x)$ be any automorphism. Since $\text{degree}(\xi)\,=\,\text{degree}((\sigma^{-1})^*\xi)$,
there is a line bundle $L_\sigma \in J(X)$ such that
$$\sigma^*(\xi \otimes L_\sigma^{\otimes r}) \,\cong\, \xi\, .$$
Moreover, if $L_\sigma'\in J(X)$ is another line bundle with the same property, then
$(L_\sigma')^{\otimes r} \,\cong\, L_\sigma^{\otimes r}$, so $L_\sigma$ and $L_\sigma'$ differ
by tensoring with an $r$-torsion element of the Jacobian $J(X)$.

Thus, if we pick a choice for $L_\sigma$ for each $\sigma\in \Aut(X,X)$, then $\ST$ is generated as a group by the maps
\begin{itemize}
\item $\overline{\ST_{\sigma,L_\sigma,+}}$ for $\sigma\in \Aut(X,x)$
\item $\overline{\ST_{\id,L,+}}$ for $L\in J(X)[r]$.
\end{itemize}

Moreover, for every $\sigma \in \Aut(X,x)$, every $L\in \Pic(X)$ and every $L'\in J(X)[r]$, we have
$$\overline{\ST_{\sigma,L,+}}\circ \overline{\ST_{\id,L',+}} =
\overline{\ST_{\id,\sigma^*L',+}}\circ \overline{\ST_{\sigma,L,+}}\,\, .$$
Since $\sigma^*:J(X)[r] \longrightarrow J(X)[r]$ is an automorphism, it follows that
$$\overline{\ST_{\sigma,L,+}} \circ J(X)[r] = J(X)[r]\circ \overline{\ST_{\sigma,L,+}}\, \,.$$
Therefore, $J(X)[r]$ is a normal subgroup of $\ST$ and its
quotient is precisely $\Aut(X,x)$, so we obtain an exact sequence
\begin{eqnarray*}
\xymatrixcolsep{4pc}
\xymatrix{
1 \ar[r] & J(X)[r] \ar[r]^-{L \mapsto \overline{\ST_{\id,L,+}}} & \ST \ar[r]^-{\overline{\ST_{\sigma,L,+}} \mapsto \sigma} & \Aut(X,x) \ar[r] & 1
}
\end{eqnarray*}
This completes the proof.
\end{proof}

\section*{Acknowledgements} 

We are very grateful to the referee for helpful comments.
This work was developed during a research stay of both authors at the Laboratoire J. 
A. Dieudonn{\'e} at Universit{\'e} de Nice Sophia-Antipolis. We would like to thank 
the laboratory for its hospitality. This research was partially funded by MINECO 
(grant MTM2016-79400-P and ICMAT Severo Ochoa project SEV-2015-0554) and the 7th 
European Union Framework Programme (Marie Curie IRSES grant 612534 project MODULI). 
The first author was also supported by a predoctoral grant from Fundaci\'on La Caixa 
-- Severo Ochoa International Ph.D. Program and a postdoctoral position associated to the Severo Ochoa project. Moreover, he would like to thank Tom\'as G\'omez for the useful discussions held during the development of this work. The second
author is supported by a J. C. Bose Fellowship.


\begin{thebibliography}{ZZZZZZZ}

\bibitem[Alf18]{Alf18}
D.~Alfaya, Automorphism group of the moduli space of parabolic vector bundles over a curve,
{\em Thesis, {Universidad Aut\'onoma de Madrid}} (2018).

\bibitem[AG19]{AG19}
D.~Alfaya and T.~G\'omez, Automorphism group of the moduli space of parabolic bundles over a curve,
{\em arXiv:1905.12404} (2019).

\bibitem[At56]{At} M. F. Atiyah, On the Krull--Schmidt theorem with application to
sheaves, {\it Bull. Soc. Math. Fr.} {\bf 84} (1956), 307--317.

\bibitem[BBGN07]{BBGN} V. Balaji, I. Biswas, O. Gabber and D. S. Nagaraj, Brauer obstruction
for a universal vector bundle, {\em Com. Ren. Math. Acad. Sci. Paris} {\bf 345} (2007), 
265--268.

\bibitem[Bh99]{Bh} U. N. Bhosle, Picard groups of the moduli spaces of vector bundles,
{\em Math. Ann.} {\bf 314} (1999), 245--263. 

\bibitem[BBN09]{BBN09}
I.~Biswas, L.~Brambila-Paz and P.~E. Newstead,
Stability of the projective {Poincar\'e} and {Picard} bundles,
{\em Bull. Lond. Math. Soc.} {\bf 41} (2009), 458--472.

\bibitem[BG08]{BG08}
I.~Biswas and T.~G\'omez, Simplicity of stable principal sheaves,
{\em Bull. Lond. Math. Soc.}, {\bf 40} (2008), 163--171.

\bibitem[BGM10]{BGM10}
I.~Biswas, T.~L.~Gomez, and V.~Mu{\~n}oz,
Torelli theorem for the moduli space of framed bundles,
{\em Math. Proc. Cambridge Phil. Soc.} {\bf 148} (2010), 409--423.

\bibitem[BGM12]{BGM12}
I.~Biswas, T.~L.~G\'omez, and V.~Mu{\~n}oz,
Automorphisms of moduli spaces of symplectic bundles,
{\em Internat. J. Math.} {\bf 5} (2012), 1250052, 27p.

\bibitem[BGM13]{BGM13}
I.~Biswas, T.~G\'omez, and V.~Mu{\~n}oz,
Automorphisms of moduli spaces of vector bundles over a curve,
{\em Expo. Math.} {\bf 31} (2013), 73--86.

\bibitem[BPGN97]{BPGN97}
L.~Brambila-Paz, I.~Grzegorczyk and P.~E. Newstead,
Geography of {Brill}--{Noether} loci for small slopes,
{\em Jour. Algebraic Geom.} {\bf 6}(4) (1997), 645--669.

\bibitem[Don84]{Don84}
S.~K. Donaldson, Instantons and geometric invariant theory,
{\em Comm. Math. Phys.} {\bf 93} (1984), 453--460.

\bibitem[Fa93]{Fa93}
G. Faltings, Stable $G$-bundles and projective connections,
{\em Jour. Algebraic Geom.} {\bf 2} (1993), 507--568.

\bibitem[Hi87a]{Hi1} N. J. Hitchin, The self-duality equations on a Riemann surface,
{\it Proc. London Math. Soc.} {\bf 55} (1987), 59--126.

\bibitem[Hi87b]{Hi2} N. J. Hitchin, Stable bundles and integrable systems, {\it
Duke Math. Jour.} {\bf 54} (1987), 91--114.

\bibitem[HL95a]{HL95modules}
D.~Huybrechts and M.~Lehn, Framed modules and their moduli,
{\em Inter. Jour. Math.} {\bf 6} (1995), 297--324.

\bibitem[HL95b]{HL95pairs}
D.~Huybrechts and M.~Lehn, Stable pairs on curves and surfaces,
{\em Jour. Algebraic Geom.} {\bf 4} (1995), 67--104.

\bibitem[KP95]{KP95}
A.~Kouvidakis and T.~Pantev, The automorphism group of the moduli space of semi stable vector
bundles, {\em Math. Ann.} {\bf 302} (1995), 225--268.

\bibitem[Ma76]{Ma} M. Maruyama, Openness of a family of torsion free sheaves,
{\em Jour. Math. Kyoto Univ.} {\bf 16} (1976), 627--637.

\bibitem[NS65]{NS65} M. S. ~Narasimhan and C. S. ~Seshadri., Stable and unitary vector bundles on a compact Riemann surface,
{\em Ann. of Math.} {\bf 82} (1965), 540--567.

\bibitem[NR75]{NR75} M. S. ~Narasimhan and S. ~Ramanan., Deformations of the moduli space of vector bundles over a curve,
{\em Ann. of Math.} {\bf 101} (1975), 39--47.

\bibitem[Ra73]{Ra} S. Ramanan, The moduli spaces of vector bundles over an
algebraic curve, {\em Math. Ann.} {\bf 200} (1973), 69--84.

\bibitem[Si94]{Si} C. T. Simpson, Moduli of representations of the fundamental group
of a smooth projective variety I, \emph{Inst. Hautes \'Etudes Sci. Publ. Math.} {\bf
79} (1994), 47--129.

\end{thebibliography}
\end{document}